\documentclass[final,onefignum,onetabnum,letterpaper]{siamonline190516}

\usepackage{lipsum}
\usepackage{amsfonts}
\usepackage{epstopdf}
\ifpdf
  \DeclareGraphicsExtensions{.eps,.pdf,.png,.jpg}
\else
  \DeclareGraphicsExtensions{.eps}
\fi

\usepackage{datetime}
\newdateformat{monthyeardate}{%
  \monthname[\THEMONTH] \THEDAY, \THEYEAR}

\usepackage{academicons}
\usepackage{xcolor}
\renewcommand{\orcid}[1]{\href{https://orcid.org/#1}{\textcolor[HTML]{A6CE39}{orcid.org/#1}}}

\usepackage{amsmath}
\allowdisplaybreaks
\usepackage{amssymb}
\usepackage{commath}
\usepackage{mathtools}
\usepackage{bbm}

\usepackage{color}
\usepackage{graphicx}
\usepackage[small]{caption}
\usepackage{subcaption}

\usepackage{relsize}
\usepackage{adjustbox}
\usepackage{algorithm}
\usepackage[noend]{algpseudocode}
\usepackage{booktabs}
\usepackage{tikz}

\usepackage{verbatim}
\usepackage{ulem}

\usepackage{enumitem}
\setlist[enumerate]{leftmargin=.5in}
\setlist[itemize]{leftmargin=.5in}

\newsiamthm{problem}{Problem}
\newsiamremark{remark}{Remark}
\newsiamremark{hypothesis}{Hypothesis} 
\crefname{hypothesis}{Hypothesis}{Hypotheses}
\newsiamremark{example}{Example}
\newsiamthm{claim}{Claim}
\newsiamthm{conjecture}{Conjecture}

\headers{Oscillation free high-order FV}{ D. Hillebrand, S.-Ch. Klein,  and P. \"Offner}

\title{Comparison to control oscillations in  high-order Finite Volume schemes via  physical constraint limiters, neural networks and polynomial annihilation\thanks{
\monthyeardate\today 
\corresponding{Philipp \"Offfner}\funding{This work was partially supported by the German Science Foundation (DFG) under Grant SO 363/15-1 (Hillebrand), Grant SO 363/14-1 (Klein) and the Gutenberg Research College, JGU Mainz (\"Offner). }
}}

\author{
Dorian Hillebrand\thanks{Institute of Mathematics, Technical University Brunswick, Brunswick, Germany,\email{d.hillebrand@tu-braunschweig.de}, 
}
\and 
Simon-Christian Klein\thanks{Institute of Mathematics, Technical University Brunswick, Brunswick, Germany, (\email{simon-christian.klein@tu-braunschweig.de}, \orcid{0000-0002-8710-9089})}
\and 
Philipp \"Offner\thanks{Institute of Mathematics, Johannes Gutenberg University, Mainz, Germany, (\email{poeffner@uni-mainz.de}, \orcid{0000-0002-1367-1917})} 
}

\DeclareMathOperator*{\argmin}{arg\,min} 
\newcommand{\scp}[2]{\left\langle{#1, #2}\right\rangle}

\newcommand{\vd}{\mathrm{d}}

\newcommand{\R}{\mathbb{R}} 
\renewcommand{\epsilon}{\varepsilon}

\renewcommand{\div}{\operatorname{div}}
\newcommand{\bu}{\mathbf{u}}

\newcommand{\fnum}{f^{\mathrm{num}}}
\newcommand{\fprec}{f^{n,\mathrm{precise}}}

\newcommand{\fnn}{f^{n, \mathrm{num}}}
\newcommand{\fnnnn}{f^{n, \mathrm{neural}}}
\newcommand{\Fnum}{F^{\mathrm{num}}}
\newcommand{\Fnn}{F^n}
\newcommand{\uinit}{\mathcal{I}}

\newcommand{\nquad}{\mathrm{I}}

\newcommand{\of}[1]{\left (#1 \right)}
\newcommand{\derive}[2] {\frac{\partial {#1} }{\partial {#2}}}
\newcommand{\derd}[2]{\frac{\vd {#1}}{\vd {#2}}}

\DeclareMathOperator{\ch}{conv}
\DeclareMathOperator{\PR}{P}
\DeclareMathOperator{\ran}{ran}

\ifpdf
\hypersetup{
  pdftitle={main_machine_learning},
  pdfauthor={Philipp \"Offner}
}
\fi

\begin{document}

\maketitle




\begin{abstract}
The construction of high-order structure-preserving numerical schemes to solve hyperbolic conservation laws has attracted a lot of attention in the last decades and various different ansatzes exist. 
In this paper, we compare  three completely different approaches, i.e. physical constraint limiting, deep neural networks and the application of polynomial annihilation to 
construct high-order oscillation free Finite Volume (FV) blending schemes. We further analyze their analytical and numerical properties. 
We demonstrate that all techniques can be used and yield highly efficient FV methods but also come with some additional drawbacks which we point out. 
Our investigation of the different blending strategies should lead to a better understanding of those techniques and can be transferred
 to other numerical methods as well which use similar ideas.

\end{abstract}


%

\section{Introduction} 
\label{sec:introduction} 

Hyperbolic conservation laws play a fundamental role within mathematical models for various physical processes, including fluid mechanics, electromagnetism and wave phenomena. However, since especially nonlinear conservation laws cannot be solved analytically, numerical methods have to be applied.
Starting already in 1950 with first-order finite difference methods (FD), the development has dramatically increased over the last decades including finite volume (FV) and finite element (FE) ansatzes \cite{richtmyer1994difference, du2016handbook, abgrall2017handbook}. To use modern computer power efficiently, high-order methods are nowadays constructed  which are used to obtain accurate solutions in a fast way. However, the drawback of high-order methods is that they suffer from stability issues, in particular after the developments of discontinuities which is a natural feature of hyperbolic conservation laws/balance laws. 
Here, first-order methods are favorable since their natural amount of high dissipation results in robust methods. In addition, many first-order methods have also the property that they preserve  other physical constraints like the positivity of density or pressure in the context of the Euler equations of gas dynamics.
In contrast, high-order approaches need additional techniques like positivity preserving limiters, etc. \cite{zhang2011maximum}.
Due to those reasons, researchers have combined low-order methods with high-order approaches to obtain schemes with favorable properties. The high-order accuracy of the method in smooth regions is kept, while also the excellent  stability conditions  and the  preservation of physical constraints   of the low order methods  near the discontinuities remain.  
Techniques in such context are e.g. Multi-dimensional Optimal Order Detection (MOOD) \cite{bacigaluppi2019posteriori, clain2011high}, subcell FV methods \cite{sonntag2014shock, hennemann2021provably} or limiting \cite{guermond2019invariant,kuzmin2020monolithic, kuzmin2021Limiter} strategies to name some.
In the last two approaches, free parameters are selected/determined which mark the problematic cells where the discontinuity may live. Here, the low order method is used whereas in the unmarked cells the high-order scheme still remains. 
To select those parameters, one uses either  shock sensors \cite{persson2006sub, offner2015zweidimensionale} or constraints on physical quantities  (entropy inequality, the positivity of density and pressure, etc.). 
As an alternative to those classical ansatzes, the application of machine learning (ML) techniques as shock sensors and to control oscillations have recently driven a lot of attention \cite{abgrall2020neural, beck2020neural, discacciati2020controlling, zeifang2021data}.  ML can be used for function approximation, classification and regression  \cite{Cybenko1989}. In this manuscript, we will extend those investigations in various ways. \\
In  \cite{klein2021using}, the author has proposed a simple blending scheme that combines a high-order entropy conservative numerical flux with the low-order Godunov-type flux in a convex combination. The approach is somehow  related to convex limiting. The convex parameter is selected by a
 predictor step automatically to enforce that the underlying method satisfies the Dafermos entropy condition numerically.  We focus on this scheme and extend the investigation from  \cite{klein2021using} in various ways.
 First, we propose a novel selection criteria not only based on Dafermos entropy criteria  \cite{dafermos1973entropy}  but rather on the preservation of other physical constraints, e.g. the positivity of density and pressure. As an alternative ansatz, we  further investigate the application of forward neural networks (NN) to specify the convex  parameter. As the last approach, we apply polynomial annihilation (PA) operators described in \cite{glaubitz2019high}. 
Our investigation of the different limiting strategies should lead to a better understanding of those techniques and can be transferred to 
alternative approaches based on similar ideas. Finally, all of our extensions will lead to highly efficient numerical methods for solving hyperbolic conservations laws.  The rest of the paper is organized as follows:\\
In  \cref{se_numerical_method}, we present the one-dimensional blending scheme from 
\cite{klein2021using}, introduce the notation and repeat its basic properties.  We further demonstrate that also a fully discrete entropy inequality will be satisfied locally under additional constraints on the blending parameter. In \cref{se_phys_cons}, we further specify the parameter selection not only taking the entropy condition into account but also other physical constraints. 
Here, we concentrate on the Euler equation of gas dynamics and demand the positivity of pressure and density as well. In \cref{se_neunet}, we repeat forward NN and how we apply them to determine the convex parameter in the extended blending scheme to obtain an oscillation free numerical scheme. In \cref{se_regul}, the polynomial annihilation  operators are finally 
explained and how they are  used in our framework to select the blending parameter.  In \cref{se_numerics}, we test all presented methods and limiting strategies and compare the results with each other. We discuss the advantages and disadvantages of all presented methods and give finally a summary with a conclusion.

\section{Numerical Method for Hyperbolic Conservation Laws}\label{se_numerical_method}

\subsection{Notation}
We are interested in solving  hyperbolic conservation laws
\begin{equation} 
\partial_t \bu(x, t) + \partial_x  f(\bu (x, t)) = 0,\quad x \in \Omega \subset \R, t>0,  \\
\label{eq:hpde}
\end{equation}
where  $\bu: \R \times \R \to \R^m$ is the conserved variables and $f$ is the flux function. In this manuscript, we restrict ourself to the one-dimensional setting for simplicity.  In case of a scalar equation, we use additional $u$ instead of $\bu$.
Equation \eqref{eq:hpde} will be later equipped with suitable boundary and initial conditions.
 Since hyperbolic conservation laws may develop discontinuities even for smooth initial data, weak solutions are considered but they are not necessarily unique.
Motivated from physics, one selects the  solution which fulfills the additional 
entropy inequality 
\begin{equation}
\partial_t  U( \bu) + \div F( \bu) \leq 0
\label{eq:eie}
\end{equation}
with convex entropy $U$ and entropy flux $F$. We are working in the framework of FD/FV methods, therefore different kinds of numerical fluxes  are used  in the paper. We denote a general numerical flux of  $f$ with $\fnum$. It has two or more arguments in the following, i.e. 
 $\fnum (\bu_{k-p+1}, \dots, \bu_{k+p})$. 
If we apply an entropy stable flux, e.g. the Godunov flux, we denote this numerical flux by  $g: \R^m \times \R^m \to \R^m$. Using $g$ in a classical  FV/FD methodology results in a low (first) order method. Contrary,  $h: \R^m \times \R^m \to \R^m$ denotes an entropy conservative and 
high-order accurate numerical flux, cf. \cite{tadmor1987numerical, lefloch2002fully}. Please be aware that $g$ and $h$ even without the superscript $\operatorname{num}$ denote always in this paper numerical fluxes. 
The entropy-entropy flux pairs $(U, F)$ are designated using uppercase letters and the notation of  numerical entropy fluxes  are the same as above. It is that the numerical entropy flux $G: \R^p \times \R^p \to \R^p$ is associated with a dissipative numerical flux $g$, analog for $h$. We use further the standard abbreviation, i.e. $g(\bu(x_k, t), \bu(x_{k+1}, t)) = g(\bu_{k}(t), u_{k+1}(t)) = g_{k+\frac 1 2}(t)
$
 generalizing $x_k$ as a way of referring to the center of cell $k$ and $x_{k+\frac 1 2}$  to the right cell boundary, cf. \cref{fig:celldivision}.
 \begin{figure}
			\begin{center}
				\begin{tikzpicture}
					\draw (-5, 0) -- (5, 0); 
					\draw (-5, 0) node [left] {$t^n$};
					\draw (-1, 0) -- (-1, 2); 
					\draw (-1, 2) node [above]{ $\fnum_{k - \frac 1 2}$};
					\draw (-1, 0) node [below] {$x_{k - \frac 1 2}$};
					\draw (1, 0) -- (1, 2); 
					\draw (1, 2) node [above]{ $ \fnum_{k + \frac 1 2}$};
					\draw (1, 0) node [below] {$x_{k + \frac 1 2}$};
					\draw [dotted](0, 0) -- (0, 2); 
					\draw (0, 2) node [above] {$f_k$};
					\draw (0, 0) node[below] {$x_k$};
					\draw (-3, 0) -- (-3, 2); 
					\draw (-2, 0) node [below] {$x_{k-1}$};
					\draw (3, 0) -- (3, 2);
					\draw (2, 0) node [below] {$x_{k+1}$};
					\draw (-2, 1) node {$u_{k-1}$};
					\draw (2, 1) node {$u_{k+1}$};
					\draw (0, 4/3) node {$u_k^n$};
					\draw (0, 2/3) node {$u_{k-\frac 1 4}\, u_{k+ \frac 1 4}$};
					\draw (-5, 2) -- (5, 2); 
					\draw (-5, 2) node [left] {$t^{n+1}$};
				\end{tikzpicture}
				\caption{The subdivision of a cell in space, initialized with the mean value of the old cell.}
				\label{fig:celldivision}
				\end{center}
			\end{figure}
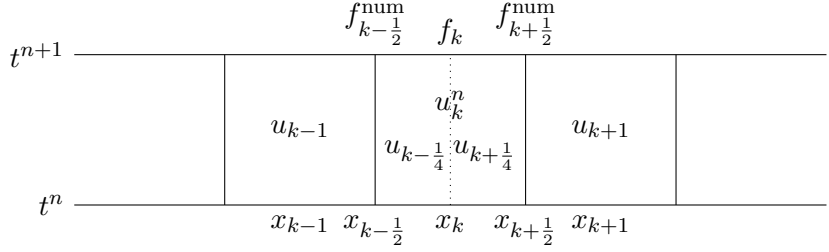
 The same procedure is used for grid points in time in the fully discrete setting, i.e. 
$
  g(\bu(x_k, t_n), \bu(x_{k+1}, t_n)) = g(\bu^n_{k}, \bu^n_{k+1}) = g_{k+\frac 1 2}^n.
$
 Please note that a $2p$ point numerical flux at position $k+ \frac 1 2$ used the points $\bu_{k-p+1}, \dots, \bu_{k+p}$, e.g. for $p=2$ we have 
$
  h(\bu^n_{k-1}, \bu^n_{k}, \bu^n_{k+1}, \bu^n_{k+2}) = h^n_{k+\frac 1 2}.
$
Convex combined numerical fluxes are written as 
 \[
 f_{\alpha_{k+\frac 1 2}}^n = \alpha_{k + \frac 1 2} g^n_{k+\frac 1 2} + \left(1-\alpha_{k+\frac 1 2}\right) h^n_{k+\frac 1 2}.
 \]
Working with reconstruction free FV methods, the numerical solution in the cell is constant in space at a certain time, in short form i.e. 
$
	f^n_k = f(u^n_k) = f(u(x_k, t_n))
$
for instance. Sometimes cells are cutted in half at position $x_k$ as described in \cref{fig:celldivision}. Therefore  it exists the cell interfaces at $x_{k-1}, x_{k-\frac 1 2}, x_k, x_{k + \frac 1 2}$ and $x_{k+1}$. The middlepoints are  $x_{k-\frac 3 4}, x_{k-\frac 1 4}, x_{k+\frac 1 4}, x_{k+\frac 3 4}$.
We use an uniform mesh with cell length $\Delta x=x_{k+ \frac 12}-x_{k - \frac 12}$ and uniform time-steps $\Delta t= t^{n+1}-t^n$. The mesh ratio is defined by $\lambda= \frac{\Delta t} {\Delta x}$.\\
As mentioned above, to select the physical meaningful solution \eqref{eq:eie} has to be fulfilled.  In terms of our numerical approximation,
the determined solution has been constructed to imitate \eqref{eq:eie} discretely, i.e. in context of first-order FV/FD this means
\[
    \frac{U^{n+1}_k - U^{n}_k}{\Delta t} + \frac{G^n_{k+ \frac 1 2} - G^n_{k-\frac 1 2}}{\Delta x} \leq 0
\]
for an entropy stable numerical entropy flux $G$. 
If the approximated solution satisfies for all entropy pairs such inequality, we call the scheme entropy stable and entropy dissipative if it is only fulfilled for one specific entropy pair. In the last years, many researchers have worked on the construction of entropy conservative and dissipative schemes based either on FD, FV or FE ansatzes, cf. \cite{abgrall2018general, abgrall2021analysis, abgrall2022reinterpretation, chan2018discretely, chen2020review, fisher2013discretely, offner2018stability, ranocha2016summation}. Here, the entropy condition was fulfilled locally.

\subsection{FV Method with Predictor-Corrector Fluxes}
To explain our blending scheme, we start from the classical FV method.  
A FV method results from integrating the conservation law over a rectangle $\left[x_{k-\frac 1 2}, x_{k + \frac 1 2}\right] \times \left[t^n, t^{n+1}\right]$
\begin{equation} \label{eq:discFVM}
\begin{aligned}
	 u^{n+1}_k =& \int_{x_{k-\frac 1 2}}^{x_{k+\frac 1 2}} \frac{u(x, t^{n+1})}{\Delta x} \vd x \\=& \int_{x_{k+\frac 1 2}}^{x_{k+\frac 1 2}} \frac{u(x, t^n)}{\Delta x} \vd x &+& \frac{1}{\Delta x}\int_{t^n}^{t^{n+1}} f\left(u\left(x_{k-\frac 1 2}, \tau\right)\right) - f\left(u\left(x_{k+\frac 1 2}, \tau \right)\right)  \vd \tau \\
	  \approx&  u^{n}_k &+& \frac{\Delta t}{\Delta x} \left(\fnn_{k-\frac 1 2} - {\fnn_{k+\frac 1 2}}\right).
	 \end{aligned}
\end{equation}
Taking the limit $\lim_{\Delta t \to 0} \frac{1}{\Delta t}$ in \eqref{eq:discFVM} results in a system of ordinary differential equations (ODEs) which can be solved 
using e.g. Runge-Kutta (RK) schemes \cite{SO1988, SO1989}.  Here, one split between the space and time discretization also referred to as the method of lines ansatz (MOL). If only the PDE is discretized in space, we call the scheme in semi-discrete form.  
 A different approach is based on the assumption that a numerical flux for timesteps $\Delta t=t^{n+1}-t^n$ could be devised based on knowledge of the conservation law and the local time evolution of the solution. 
Based on this line of thought is the Cauchy Kowaleskaya expansion used in \cite{ENOIII} to provide a high-order time-stepping method.  
The drawback of the Cauchy Kowaleskaya approach is that it typically results in lengtly calculations, complex implementations and/or implicit methods where nonlinear solvers are needed. However, we distinguish between the semidiscrete and the fully discrete schemes in the following. 
Obviously, in  \eqref{eq:discFVM} the coupling between neighboring cells has been done via numerical fluxes $\fnn_{k-\frac 1 2}$ to ensure the conservation property. 
A vast amount of numerical fluxes is known in the literature \cite{Lax71, Roe1981, HLL1983, IsmailRoe2009} and even selecting a flux is a nontrivial task \cite{ranocha2018comparison}. Some fluxes, like the Godunov, Lax-Friedrichs, Roe and HLL fluxes, that can be interpreted by exact or approximate Riemann problem solutions, are meant to approximate the flux through some cell boundary over time $\Delta t$, i.e. being the mean value of the flux over this period.  
Numerical fluxes that have only a semidiscrete interpretation need some sort of high order time integration method. 
Our method for high-order time integration is based on a reinterpretation of predictor-corrector time integration \cite[p. 386]{Isaacson1966Analyis} as a numerical quadrature of the numerical flux over a cell boundary.

\begin{theorem}{(Predictor-Corrector-Fluxes)}\label{thm:rkflux}
	Let $\fnum(u_{k}, u_{k+1})$ be a numerical flux and  $u_k(t)$ on $[t, t + \Delta t]$ be the exact solution of the scheme 	\[
	\derd{u_k(t)}{t} + \frac{\fnum(u_{k}(t), u_{k+1}(t)) - \fnum(u_{k-1}(t), u_{k}(t))}{\Delta x} = 0
	\]
	with uniform cell size $\Delta x$.  
	Then, the 4-point numerical flux $\fnum(u_{k-1}, u_{k}, u_{k+1}, u_{k+2})$ defined as
	\[
	\begin{aligned}
	u^1_k =& u_k + \lambda (\fnum(u_{k-1}, u_k) - \fnum(u_{k}, u_{k+1})), \\
	u^1_{k+1} =& u_{k+1} + \lambda (\fnum(u_k, u_{k+1}) - \fnum(u_{k+1}, u_{k+1})), \\
	\fnum(u_{k-1}, u_{k}, u_{k+1}, u_{k+2}) =& \frac{\fnum(u_k, u_{k+1}) + \fnum\of{u^1_{k}, u^1_{k+1}}}{2}
	\end{aligned}
	\]
	is a second-order\footnote{Please note that the term $p$ order accurate was coined so that integration via a $p$ order quadrature rule leads to a $p$ order accurate approximation.} accurate approximation of
$
		\frac{1}{\Delta t}\int_{t}^{t + \Delta t} \fnum(u_{k}(\tau), u_{k+1}(\tau)) \vd \tau,
$
	i.e.
	\[
		\norm{\fnum(u_{k-1}, u_{k}, u_{k+1}, u_{k+2}) - \frac{1}{\Delta t}\int_{t}^{t + \Delta t} \fnum(u_{k}(\tau), u_{k+1}(\tau)) \vd \tau} = \mathcal O{(\Delta t)^2}.
	\]
	\begin{proof}
		We begin by stating that the intermediate values $u^1_{k}, u^1_{k+1}$ are first-order accurate, i.e. 
		\[
		u^1_{k} = u_k(t + \Delta t) + \mathcal O\of{(\Delta t)^2} \quad u^1_{k+1} = u_{k+1}(t + \Delta t) + \mathcal O\of{(\Delta t)^2}
		\]
		due to the explicit Euler method. Calculation of the flux between cell $u_k$ and $u_{k+1}$ over time $\Delta t$ via the trapezoid rule $\nquad$ (second-order) and the exact solution $u_k(t)$ is second-order accurate, i.e.
		\[
		\begin{aligned}
		\nquad[\fnum(u_{k}(\cdot), u_{k+1}(\cdot))] = &\frac{\Delta t}{2}(\fnum(u_k(t), u_{k+1}(t)) + \fnum(u_{k}(t + \Delta t), u_{k+1}(t+ \Delta t))) \\ = &\int_t^{t + \Delta t} \fnum(u_{k}(\tau), u_{k+1}(\tau))\vd \tau + \mathcal O\left((\Delta t)^3\right).
		\end{aligned}
		\]
		Due to the Lipschitz continuity of $\fnum$, we have 
		\[
		\norm{\fnum(u_l(t + \Delta t), u_r(t + \Delta t)) - \fnum(u^1_l, u^1_r)}\leq L_f \left(\norm {u_l(t + \Delta t) - u^1_l} + \norm{u_r(t + \Delta t) - u^1_r}\right),
		\]
		where $u_l$ and $u_r$ denote the left and right value at some generic interface. 
	Due to the accuracy order of $u^1_k$ and $u^1_{k+1}$, it follows
		\[
			\norm{\fnum(u_k(t + \Delta t), u_{k+1}(t + \Delta t)) - \fnum(u^1_k, u^1_{k+1})} = \mathcal O\left(\Delta t^2 \right).
		\]
		The combination of these three statements yields that the numerical quadrature of the flux calculated using the approximate values $u^1_k, u^1_{k+1}$
		\[
		\begin{aligned}
		\Delta t \fnum 	=& \frac{\Delta t }{2}(\fnum(u_{k}, u_{k+1}) + \fnum(u^1_k, u^1_{k+1})) \\
						=& \frac{\Delta t }{2}(\fnum(u_k, u_{k+1}) + \fnum(u_k(t + \Delta t), u_{k+1}(t + \Delta t)) + \mathcal O(\Delta t)^2) \\
						=& \nquad[\fnum(u_k(\cdot), u_{k+1}(\cdot))] + \mathcal O(\Delta t)^3 \\
						=& \int_t^{t + \Delta t} \fnum(u_k(\tau), u_{k+1}(\tau))\vd \tau + \mathcal O\left(\Delta t^3\right) 
		\end{aligned}
		\]
		is a second-order exact approximation and dividing by $\Delta t$ induces the result.
		\end{proof}
\end{theorem}
The above numerical flux $\fnum(u_{k-1}, u_{k}, u_{k+1}, u_{k+2})$ could be also interpreted as the flux over the given cell boundary if the semidiscrete scheme is used together with the strong stability preserving (SSP) RK(2,2) method which is equivalent to the deferred correction method of order 2 \cite{abgrall2021relaxation}. However, higher-order quadrature rules can also be applied in this context. 
%
To describe now the method, we follow  \cite{klein2021using} where the considered blending FV scheme has been proposed. The method  fulfills  Dafermos' entropy condition   \cite{dafermos1973entropy} as defined like follows: 
\begin{definition}[Dafermos' Criteria]\label{def_Dafermos}
Let $\bu$ be a weak solution of \eqref{eq:hpde} and $U$ an entropy. The total entropy in the domain  $\Omega$ is given by 
\[
E_\bu(t) = \int_{\Omega} U( \bu(x, t)) \vd x.
\]
A Dafermos entropy solution $\bu$ is a weak solution that satisfies 
\begin{equation}\label{eq_Dafermos}
 \forall t > 0: \quad \partial_t E_\bu(t) \leq \partial_t E_{\tilde{\bu}}(t) 
\end{equation}
compared to all other weak solutions $\tilde \bu$ of the conservation law \eqref{eq:hpde}. In essence, the entropy of the selected solution decreases faster than the entropy of all other solutions.
\end{definition}

\begin{definition}\label{def_Blending}
The blending scheme is based on the FV approach in conservative form. Instead of using classical numerical fluxes in 
\eqref{eq:discFVM}, a convex combination between classical Godunov-type flux  and  a high-order entropy conservative flux  is used instead. 
The combined flux, called GT-flux, is given by
	\begin{equation}\label{eq_GT_flux}
     f_{\alpha_{k+\frac 1 2}}^n := \alpha_{k + \frac 1 2} g^n_{k+\frac 1 2} ( {\bu_k,\bu_{k+1}}) + (1-\alpha_{k+\frac 1 2}) h^n_{k+\frac 1 2} (\bu_{k-p+1}, \dots, \bu_{k+p})
     	\end{equation}
	where $\alpha_{k+\frac 1 2} \in [0,1]$ is the convex parameter. 
\end{definition}
\begin{example}
To give a concrete example, using the explicit Euler method for the time,  the scheme is given by 
\begin{equation}\label{eq_Scheme}
\bu^{n+1}_i =\bu^{n}_i + \frac{\Delta t}{\Delta x} \left(  f_{\alpha_{k+\frac 1 2}}^n   - f_{\alpha_{k-\frac 1 2}}^n  \right).
\end{equation}
To obtain higher order in time, RK methods can be used instead, which can be rewritten into the fluxes as described in  \cref{thm:rkflux} yielding a high-order method that can be written with explicit Euler steps. Please be aware, that also different time-integration methods can be combined for the convex combination resulting in highly efficient schemes using  \cref{thm:rkflux}.
\end{example}
The properties of the scheme highly depend on the selected fluxes and  on the convex parameter $\alpha_{k+\frac12}$.
The value of $\alpha_{k+\frac12} = \alpha (\bu_{k-p+1}, \dots, \bu_{k+p})$ itself depends on $\bu_i$ 
which takes the high-order stencil into account.
Before we repeat how $\alpha$ has to be selected to ensure that our scheme fulfills additionally Dafermos' criteria \eqref{eq_Dafermos}, we want to summarize the following  basic properties of the scheme and the numerical fluxes:
\begin{itemize}
\item The GT-flux is consistent and local Lipschitz continuous \cite[Lemma 1]{klein2021using}. 
\item The GT-flux with Tadmor's entropy conservative flux or the high-order modification from \cite{lefloch2002fully} satisfies as well the semidiscrete cell entropy inequality locally for the selected entropy pair used in the construction of the flux for all $\alpha  \in (0, 1]$ \cite[Theorem 1]{klein2021using}.
\item Due to the conservation form of \eqref{eq_Scheme} and the convex combination of the flux, the scheme is locally conservative and the Lax-Wendroff theorem is valued due to the applications of the results from \cite{shi2018local}.
\end{itemize}
As we mentioned before, the parameter $\alpha$ is essential for the properties of the underlying method and we repeat from \cite{klein2021using} the following definition:
\begin{definition}
	We call $\alpha: \R^{2p\times m} \to [0, 1]$ an entropy inequality predictor with a $(2p)$ point stencil if
	\begin{align*}
&\lim_{\Delta x \to 0}	\alpha(u_{k-p+1}, \dots, u_{k +p}) \\ = &\begin{cases} 0 &  \exists x \in [x_k -(p-1)\Delta x, x_k + p\Delta x]: \derive{U}{t} + \derive{F}{x} < 0\\
	1 & \forall x \in [x_k -(p-1)\Delta x, x_k + p\Delta x]:\derive{U}{t} + \derive{F}{x} = 0 \end{cases}
	\end{align*}
	holds for the complete stencil. We will call the entropy inequality predictor slope limited if 
	\[ \abs{\alpha_k - \alpha_{k+1}} < M \quad \text{with} \quad \alpha_k = 	\alpha(u_{k-p+1}, \dots, u_{k+p} )\]
	holds for some $M < 1$ and all $i$.
\end{definition}
In \cite{klein2021using}, a slope entropy inequality predictor was constructed starting from Godunov-type flux and demonstrated  that it is slope limited. The predictor is given by 
	\begin{equation}\label{eq_predictor}
	\alpha^n = H_{sm}\left( \frac{\frac {s^n_k}{s_{ref}} -a}{b}   \right) \circledast \hat{h},
	\end{equation}
where $s^n_k$ is the entropy dissipative rate from the classical Godunov scheme
	\begin{equation*}
		s^n_k(t) = \frac{G\of{u^n_{k+1}, u^n_{k}} - G\of{u^n_k, u^n_{k-1}}} { \Delta x} + \frac{U\of{u_k^{n+1}} - U\of{u^{n}_{k}}} { \Delta t}.
	\end{equation*}
 $s_{ref}$ its  minimum value, $\hat{h}$ the cut hat function ($h(x)=\max(0,\min(1,2x+2,-2x+2)$), $H\in C^2$ the smothstep function  
		\begin{equation*}
			H_{sm}(x) = \begin{cases}
				0 & x  \leq 0 \\
				6x^5 - 15x^4+10x^3 & 0 \leq x  \leq 1\\
				1 & 1 \leq x, \\
			\end{cases}
		\end{equation*}
and $\circledast $ denotes the discrete sup-mollification 
$
		(f \circledast g)|_{[i/n, (i+1) / n]} = \max_{j \in \{0, \dots, n-1\}} f_j g_{i-j}$ for $ i = 0, \dots, n-1
$
and step functions $f,g$. 
\begin{remark}
Instead of working with the classical Godunov flux in \cref{def_Blending}, we  use approximated Riemann solvers. In \cite{klein2021using}, the local Lax-Friedrich flux  (LLF) (Rusanov) has been used and shows promising results. In the numerical section, we apply always the LLF flux to obtain a more efficient method. 
Finally, we like to point out that extension of the method to two dimension is straightforward via a tensor-structure ansatz. All of our now considered results can be transferred. 
\end{remark}

\subsection{Local entropy inequality}\label{subsec_local_entropy}
Here, we want to extend the investigation of  \cite{klein2021using} and demonstrate that the method satisfies locally a fully discrete entropy inequality 
%
under additional restrictions on $\alpha$. We  define the entropy production of the semidiscretly entropy conservative flux scheme
\[
	p^n_k = \frac{H\of{u_{k - p+1}, \dots, u_{k+p}} - H\of{u_{k-p+1}, \dots,  u_{k+p}}}{\Delta x} + \frac{U\of{u^{n+1}_k} - U \of {u^n_k}}{\Delta t}
\] 

and call  \textbf{Condition $F$}  the following:
\begin{definition}
	The parameter $\alpha$ is said to satisfy \textbf{Condition $F$} for cell $k$ if
\[
	\alpha_{k+\frac 1 2} s^n_{k + \frac 1 4} + \left(1-\alpha_{k + \frac 1 2}\right) p^n_{k + \frac 1 4} \leq 0 \text{ and }  \alpha_{k-\frac 1 2} s^n_{k - \frac 1 4} + \left(1-\alpha_{k - \frac 1 2}\right) p^n_{k - \frac 1 4} \leq 0
\]
	holds for the left and right cell interfaces.
\end{definition}
We can prove: 
\begin{lemma}\label{lem:EalpBound}
	\textbf{Condition $F$} is fulfilled for cell $k$ if one of the following conditions is satisfied on each cell interface, i.e. for $k+\frac 1 4$ and $k-\frac 1 4$.
	\begin{enumerate}
		\item It holds $s^n_{k+\frac 1 4} \leq 0$ and $p^n_{k+\frac 1 4} \leq 0$, $\alpha \in [0, 1]$ is arbitrary.
		\item It holds $s^n_{k + \frac 1 4} \leq 0$ and $p^n_{k + \frac 1 4} > 0$ and $\alpha \geq \frac{p^n_{k+\frac 1 4}}{p^n_{k + \frac 1 4} - s^n_{k + \frac 1 4}}$.
	\end{enumerate}
	\begin{proof}
		The first condition is obvious. For the second one the following calculation
		\[
			\alpha s + (1-\alpha) p \leq 0 \iff \alpha(s-p) + p \leq 0 \iff  \alpha \geq \frac{p}{p-s} \geq 0
		\]
		with suppressed indices show the result. Please note that $s-p < 0$ holds by construction.
		\end{proof}
	\end{lemma}
	The aforementioned result allows to calculate a lower bound on $\alpha$ to enforce  \textbf{Condition $F$} if one can guarantee that $s^n_{k+\frac 1 4} > 0$ never happens. As the Godunov flux or Lax-Friedrichs flux are two entropy stable examples  the case $s^n_{k + \frac 1 4}>0$ never happens if an appropriate CFL condition is minded. The following theorem is based on \cite{Tadmor84II} and \cite{klein2021using}:
\begin{theorem}\label{thm:deie}
	The discrete GT-Scheme satisfies a discrete per cell entropy inequality with the flux
	\[
	\Fnum(u_{k-p+1}, \dots, u_{k+p}) = \alpha_{k+\frac 1 2} G(u_{k-p+1}, \dots, u_{k+p}) + (1-\alpha_{k+\frac 1 2}) H(u_{k-p+1}, \dots, u_{k+p}) 
	\]
	if $\alpha_{k+\frac 1 2}$ fulfills the condition $F$ for both cell boundaries and the CFL restriction is half that of the minimum of either fluxes. 
\end{theorem}
		\begin{proof}
	
	We first state that the cell mean $u^{n+1}_k$ can be written as the average value 
	\[
	\begin{aligned}
	u^{n+1}_k 
	=& u^n_k + \lambda \left (\fnn_{\alpha_{k-\frac 1 2}} - \fnn_{\alpha_{k+\frac 1 2}} \right) \\
	=& \frac{u^n_k + 2 \lambda \left(\fnn_{\alpha_{k-\frac 1 2}} - f\of{u^n_k} \right) + u^n_k + 2 \lambda \left(f\of{u^n_k}-f^n_{\alpha_{k+\frac 1 2}} \right) } {2} 
	\\
	=& \frac{\alpha_{k-\frac 1 2} \left (u^n_k + 2 \lambda \left(g^n_{k-\frac 1 2} - f\of{u^n_k} \right) \right) 
	+ \left(1- \alpha_{k-\frac 1 2} \right)\left(u^n_k + 2 \lambda \left(h^n_{k-\frac 1 2} - f\of{u^n_k} \right)\right) }{2}\\
	&+ \frac{\alpha_{k+\frac 1 2} \left(u^n_k + 2 \lambda \left(f\of{u^n_k}-g^n_{k+\frac 1 2} \right)\right) 
	+ \left(1-\alpha_{k+\frac 1 2}\right)\left(u^n_k + 2 \lambda \left(f\of{u^n_k}-h^n_{k+\frac 1 2} \right) \right) } {2} \\
	=& \frac{u^{n+1}_{k-\frac 1 4} + u ^{n+1}_{k+\frac 1 4}}{2}
	\end{aligned}
	\]
	of two schemes and therefore one concludes that for the entropy of cell $k$ holds 
	\begin{equation}
        \begin{aligned}
		&U(u^{n+1}_k) - U(u^n_k) + \lambda (\Fnn_{\alpha_{k+\frac 1 2}} - \Fnn_{\alpha_{k - \frac 1 2}}) \\
		\leq &
		 \frac {U\of {u^{n+1}_{k-\frac 1 4}} + U \of {u^{n+1}_{k+\frac 1 4}} }{2} - U\of{u^n_k} + \lambda \left( \Fnn_{\alpha_{k+\frac 1 2}} - \Fnn_{\alpha_{k - \frac 1 2}}\right)\\
		= &
		\frac{U\of {u^{n+1}_{k-\frac 1 4}} - U \of{u^n_k} + 2 \lambda \left( F^n_{k} - \Fnn_{\alpha_{k - \frac 1 2}}\right)}{2} 
		+ \frac{U\of {u^{n+1}_{k+\frac 1 4}} - U \of{u^n_k} + 2 \lambda \left( \Fnn_{\alpha_{k+\frac 1 2}} - F^n_{k} \right)}{2} \\
		\leq & \frac{\alpha_{k-\frac 1 2}}{2}\left(U \left(u^n_k + 2 \lambda \left(g^n_{k-\frac 1 2}-f\of{u^n_k} \right) \right) - U(u^n_k) + 2 \lambda  \left( F^n_{k} -G^n_{k-\frac 1 2} \right) \right) \\ 
		&+ \frac{1-\alpha_{k-\frac 1 2}}{2} \left(U\of{u^n_k + 2 \lambda \left(h^n_{k-\frac 1 2}-f\of{u^n_k} \right)}  -U(u^n_k) + 2 \lambda  \left(F^n_{k} - H^n_{k-\frac 1 2} \right) \right) \\
		&+ \frac{\alpha_{k+\frac 1 2}}{2}\left(U \of{u^n_k + 2 \lambda \left(f\of{u^n_k}-g^n_{k+\frac 1 2} \right)} - U(u^n_k) + 2 \lambda   \left(   G^n_{k+\frac 1 2} - F^n_{k}\right) \right) \\ 
		&+ \frac{1-\alpha_{k+\frac 1 2}}{2} \left(U\of{u^n_k + 2 \lambda \left(f\of{u^n_k}-h^n_{k+\frac 1 2} \right)} - U(u^n_k) + 2 \lambda  \left(  H^n_{k+\frac 1 2} - F^n_{k}\right) \right) \\
		=& \frac{\alpha_{k-\frac 1 2}}{2}s^n_{k-\frac 1 4} + \frac{1 - \alpha_{k - \frac 1 2}}{2} p^n_{k-\frac 1 4} + \frac{\alpha_{k+\frac 1 2}}{2}s^n_{k + \frac 1 4} + \frac{1-\alpha_{k+\frac 1 2}}{2} p^n_{k+\frac 1 4} \leq 0
	\end{aligned}
	\end{equation}
	because $U$ is convex.
\end{proof}
If one can enforce \textbf{Condition $F$}, we obtain a fully discrete entropy dissipative scheme by choosing an appropriate $\alpha$. Note that the bound is  sufficient but not necessary.  As we will see in \cref{se_numerics} later, a scheme with $\alpha$ below the given bound can still be entropy dissipative locally  and  give even  better results. This  is  one of the reasons why we will investigate the usage of neuronal networks and polynomial annihilation  for the parameter selection.
But before, we focus on the Euler equation of gas dynamics. There, we can apply the same technique to enforce also the positivity of pressure and density.
Note that the above proof works for any combination of a discrete entropy stable flux in the position of the Godunov-type flux and another entropy conservative flux in the position of the Tadmor flux. It is not needed that these fluxes are two-point fluxes and it is therefore possible to use also the high-order in time fluxes from  \cref{thm:rkflux}. Especially,  the quadrature in \cref{thm:rkflux} can be adjusted to the considered fluxes, i.e. $g$ could be integrated with an SSP time integrator while $h$ could use an arbitrary high-order method for instance.

\section{Positivity of Pressure and Density  for the Euler Equations}\label{se_phys_cons}
We saw that the GT-scheme, if $\alpha$ satisfies \textbf{Condition $F$}, in turn satisfies a discrete per cell entropy inequality. The same mechanism will be used in the following paragraph to enforce positive pressure and/or density for numerical solutions of the Euler equations of gas dynamics \cite{Harten83b}:
\begin{equation}
	\label{eq:Euler}
	u =(\rho, \rho v, E) 
\quad f(\rho, \rho v, E) = \begin{bmatrix} \rho v \\ \rho v^2 + p\\ v(E + p) \end{bmatrix} 
\quad p = (\gamma - 1)\left (E - \frac 1 2 \rho v^2 \right).
\end{equation}
	This system is equipped with the following entropy-entropy flux pair 
\begin{equation} \label{eq:EulerPE}
U(\rho, \rho v, E) = - \rho S  \quad F(\rho, \rho v, E) = - \rho v S \quad S = \ln(p \rho^{- \gamma}).
\end{equation}
One easily validates that while the density is one of the conserved variables and hence a linear functional of the conserved variables the pressure is a concave functional. 
But as we are interested in preserving the positivity one can equally enforce the negativity of the negative pressure and negative density, making the task equivalent to enforcing upper bounds on convex functionals. It is known from literature, cf.  \cite{zhang2011maximum} and references therein, that  both  Godunov and Lax-Friedrichs schemes are positivity preserving under a CFL condition. In the following $g$ will stand for the flux of a positivity preserving dissipative scheme and our convex functionals that should be enforced are denoted by $c_1(u) = -p(u)$ and $c_2 = -\rho(u)$. The counterparts of \textbf{Condition $F$} are \textbf{Condition $\rho$} and \textbf{Condition $P$}.
\begin{definition}[\textbf{Condition $\rho$}]
	The parameter $\alpha$ is said to satisfy \textbf{Condition $\rho$} for cell $k$ if 
	\[
		\begin{aligned}
	c_2\of{u^n_k + 2\lambda \left(	\fnn_{\alpha_{k-\frac 1 2}} + f(u^n_k)\right) } \leq 0 \quad
	 c_2\of{u^n_k + 2\lambda \left( f^n_{k} - \fnn_{\alpha_{k+\frac 1 2}}\right) } \leq 0
		\end{aligned}
	\]
	holds.
\end{definition}

\begin{definition}[\textbf{Condition P}]
	The parameter $\alpha$ is said to satisfy \textbf{Condition P} for cell $k$ if 
	\[
	\begin{aligned}
	c_1\of{u^n_k + 2\lambda \left(	\fnn_{\alpha_{k-\frac 1 2}} + f(u^n_k)\right)} \leq 0 \quad
	 c_1\of {u^n_k + 2\lambda \left(f^n_{k} - \fnn_{\alpha_{k+\frac 1 2}}\right)} \leq 0
	 \end{aligned}
	\]
	holds
	\end{definition}
We note that the equivalent of lemma \ref{lem:EalpBound} is in this case
\begin{lemma}
	\textbf{Condition P}  is fulfilled if one of the following conditions is satisfied for $k+\frac 1 4$ and $k-\frac 1 4$.
	\begin{enumerate}
		\item It holds $c_1\of{u^n_k + 2\lambda \left(h^n_{k+\frac 1 2} - f^n_k\right)} \leq 0$ and $\alpha \in [0, 1]$ is arbitrary.
		\item It holds $c_1\of{u^n_k + 2\lambda \left(h^n_{k+\frac 1 2} - f^n_k\right)} > 0$ and $\alpha \geq \frac{c_1\of{u^n_k + 2\lambda \left(h^n_{k+\frac 1 2} - f^n_k\right)}}{c_1\of{u^n_k + 2\lambda \left(h^n_{k+\frac 1 2} - f^n_k\right)} - c_1\of{u^n_k + 2\lambda \left(g^n_{k+\frac 1 2} - f^n_k\right)}}$.
	\end{enumerate}
	\begin{proof}
		The first condition follows from the fact that for the dissipative scheme holds
		\[
		c_1\of{u^n_k + 2\lambda \left(g^n_{k-\frac 1 2} - f^n_k\right)} \leq 0
		\]
		and therefore one concludes
		\[
			\begin{aligned}
		c_1 &\of{u^n_k + 2\lambda (f^{GT}_{\alpha_{k-\frac 1 2}} - f^n_k)} \\
		= &c_1\of{\alpha\left(u^n_k + 2\lambda \left(g^n_{{k-\frac 1 2}} - f^n_k\right) \right) + (1-\alpha) \left(u^n_k + 2\lambda \left(h^n_{{k-\frac 1 2}} - f^n_k}\right)\right) \\
		\leq &\alpha c_1\of{u^n_k + 2\lambda \left(g^n_{k-\frac 1 2} - f^n_k\right)} + (1-\alpha)c_1\of{ u^n_k + 2\lambda \left(h^n_{k-\frac 1 2} - \fnn_k\right)} \leq 0
			\end{aligned}
		\]
		because both values in the convex combination on the right hand side are smaller than zero. The same holds for $\alpha_{k+\frac 1 2}$ The second part is obviously just a statement that implies the second line of the above statement.
	\end{proof}
\end{lemma}

\begin{lemma}
	If \textbf{Condition P}  is satisfied holds
	\[
		p\of{u^n_k + \lambda \left(\fnn_{\alpha_{k-\frac 1 2}} - \fnn_{\alpha_{k+\frac 1 2}} \right)} \geq 0,
	\]
	whenever the CFL restriction $\lambda c_{max} < 0.5$ is respected.
	\begin{proof}
		The convexity implies as before
		\[
		\begin{aligned}
			-p\of{u^{n+1}_k} = &c_1\of{u^{n+1}_k} = c_1\of{u^n_k + \lambda\left(\fnn_{\alpha_{k- \frac 1 2}} - \fnn_{\alpha_{k+\frac 1 2}}\right)} \\
			= &c_1\of{\frac{u^n_k + 2 \lambda \left (\fnn_{\alpha_{k-\frac 1 2}} - f^n_k\right) + u^n_k + 2\lambda\left(f^n_k - \fnn_{\alpha_{k+\frac 1 2}}\right)  }{2}} \\
			\leq &\frac 1 2 \left (c_1 \of {u^n_k + 2 \lambda \left(\fnn_{\alpha - \frac 1 2} - f^n_k\right)} + c_1 \of {u^n_k  + 2\lambda\left(f^n_k - \fnn_{\alpha + \frac 1 2}\right)} \right) \\ 
			\leq& 0.
			\end{aligned}
		\]
		\end{proof}
	\end{lemma}
	The result for \textbf{Condition $\rho$} is obviously the same and this results apply for every convex functional that is bounded by the low order scheme used in the construction. \\
Conditions on the physical constraints are not new and used in many approaches. We like to recommend the following papers from the literature which uses at least from our point of view similar ideas \cite{rueda2022subcell, kuzmin2020monolithic, kuzmin2021Limiter}.
\section{Limiting via Neural Networks} \label{se_neunet}

\subsection{Basics of Feedforward Networks}

In this section, we explain how we select our numerical flux using  feed forward neural networks (FNN).
The network is used to determine the local  indicator $\alpha$ which 
steers our convex combination inside the numerical flux and measures the regularity. This is an example of 
a high dimensional function interpolation. Our optimism concerning this problem stems from the following result proved in \cite{Cybenko1989}:
\begin{theorem}
	Let $\sigma: \R \to \R$ a sigmoidal function. Then the finite sum of the form
	\[
			\mathcal A \circ G(x) = \sum_{j=1}^N \alpha_j \sigma(\scp{y_j}{x} + b_j)
	\]
	are dense in $C(I_n)$ in the sup norm.
\end{theorem}
This theorem motivates the usage of FNN to approximate any function 
\begin{equation}
\mathcal C \subset C(\R^n, \R)
\label{eq:function_network}
\end{equation}
 on a constrained subset of $\R^n$. We will therefore give a short presentation of the general theory of neuronal networks.
Our feed forward network (also called 
multilayer perceptron (MLP)) is on particular example and it is set up in a sequence of layers containing a certain amount of neurons (computing units). 
The first layer (source layer) is handling the input data/signal to the network whereas the output layer (last layer) translates the new solution back. 
In between hidden layers are placed where the calculations are done. An FNN  with depth $K$ contains $K-1$ hidden layers and one output layer.
What happens in the network it the following operation: For an input signal $\mathbf{X} \in \R^n$, we have the output:
\begin{equation}\label{eq_function_network_2}
\mathbf{\tilde{Y}} = \mathcal{F} \circ G_k \circ \mathcal{A} \circ G_{k-1} \mathcal{A} \circ G_{k-2} \circ \cdots \circ G_1(\mathbf{X}) ,
\end{equation}
where $G_k$ denotes the affine transformation of the $k-$layer on a vector $\mathbf{Z}\in \R^{N_{k-1}}$ with 
\begin{equation}\label{eq:G_k}
G_k(\mathbf{Z}) = \mathbf{W}_k \mathbf{Z} + \mathbf{b}_k, \qquad \mathbf{W}_k  \in \R^{N_k \times N_{k-1}}, \quad \mathbf{b}_k \in \R^{N_k}.
\end{equation} 
$\mathbf{W}_k$ are the weights matrices and $\mathbf{b}_k$ are the bias vectors. Both contain the trainable parameters which we specify in the following. 
Further, in \cref{eq_function_network_2}, $\mathcal{A}$ are non-linear activation functions and  $\mathcal{F} $ is a non-linear output function that transforms 
the output data into a for us suitable form. There exists a bench of different activation functions for various problems. In our work, 
we restrict ourself to the currently popular Exponential Linear Units  (ELU) function 
\begin{equation}
\mathrm{ELU}(t) = \begin{cases}x, & x > 0, \\ \gamma (\exp(x)-1), & \text{else.} \end{cases}
\end{equation}
 We set $\gamma\equiv 1$ in our numerical simulations.\\
To approximate finally \eqref{eq:function_network} with our network \eqref{eq_function_network_2}, we must train the parameters using our training data. 
Therefore, we first create a set of training data with $N_T$ samples 
$$
T= \left\{ (\mathbf{X}_i, \mathbf{Y}_i): \mathbf{Y}_i=\mathcal{C}(\mathbf{X}_i) \forall i=1, \dots, N_T \right\}.
$$
Then, we define a suitable cost function that measures the discrepancy between the actual result vector $\mathbf{Y}$ 
and the predicted result vector $\mathbf{\tilde{Y}}$. We tested the following three cost functions:

\begin{enumerate}
\item Mean square error
$
	L(Y, \tilde Y) = \frac{\sum_{i=1}^{N_T} (Y_i - \tilde Y_i)^2}{N_T},
$
\item  Mean exponential error $L(Y, \tilde Y) = \frac{\sum_{i=1}^{N_T} (\exp(Y_i - \tilde Y_i) - 1)^2}{N_T}$,
\item 
the nonsymetric loss
$
	L(y, \tilde Y) = \frac {\sum_{i=1}^{N_T}d(Y, \tilde Y) }{N_T} \text{ with }  d(Y, \tilde Y) = \begin{cases} \gamma (Y_i - \tilde Y_i)^2 & Y_i \geq Y_i \\ \abs{Y_i - \tilde Y_i} & Y_i \leq Y_i \end{cases}
$
 and $\gamma = 10$. It gives the absolute error if the prediction is bigger than the result and the squared error if the prediction is smaller than the desired value.
\end{enumerate} 
Note that usually the second function results in higher penalties for under-prediction and that the third allows for sparse output in the domain.
 To train the network, we minimize the loss function with respect to the  parameters $\{ \mathbf{W}_k, \mathbf{b}_k \}_k$ over the set of training data. For the minimization process, we use an iterative optimization algorithm. 
We use the ADAM minimizer \cite{kingma2017adam}, alternatively we can apply the  momentum gradient descent. Both are used in the stochastic gradient descent fashion \cite{Rumelhart1986Learning}.

\begin{remark}[Overfitting and Dropout Layer]
As mentioned \textit{inter alia } in \cite{discacciati2020controlling}, the training set has to be selected quite carefully to avoid over-fitting. In such a case, the network performs poorly on general data since it is highly optimized for the training set.  To avoid this problem, a regularization technique is used. 
A popular regularization strategy is using a drop-out layer \cite{srivastava2014dropout}.
 During each optimization update step in the training-phase of the network, a dropout layer in front of the kth layer randomly sets a predefined fraction of the components of the intermediate vector computed by the kth  layer to zero. The advantages of this technique  are that the training is not biased towards a specific network architecture, additional stochasticity is injected into the optimization process to avoid getting trapped in local optima, and a sparsity structure is introduced into the network structure.
\end{remark}

\subsection{Data Driven Scheme for Single Conservation Laws and Systems}
 Our method using neuronal nets will be based on the fully discrete approach, but we will use neuronal nets as building blocks to approximate unknown real maps in the following roadmap:
\begin{enumerate}
	\item Select a random set of initial conditions $\uinit = \{u_1, u_2 , \dots, u_N \} \subset C^1_s$.
	\item Calculate high quality numerical solutions $v(x, t)$ to this set of initial conditions.
	\item Determine Projections $u$ of these solutions $v$ to a low resolution finite volume mesh.
	\item Calculate the exact flux of $v(x, t)$ over the given mesh boundaries and suitable \textbf{ time interval}.
	\item Infer suitable values for the convex combination parameter $\alpha$ \textbf{for the GT flux \eqref{eq_GT_flux}}. 
	\item Use this database to train a neuronal network as a predictor for the unknown map $\alpha(u, \Delta t)$.
\end{enumerate}
The high quality numerical solution $v$ was calculated using classic finite volume methods on fine grids. The projection of these solutions to a low resolution mesh is given by
\[
	u_k = \frac{1}{\Delta x}\int_{x_{k-\frac 1 2}}^{x_{k+\frac 1 2}} v(x, t) \vd x \quad \text{with} \quad v(x, t) = \sum_k v_k(t)\chi_{\omega_k}(x),
\] 
where $\omega_k$ shall be cell $k$ of the fine grid and $v_k$ the mean value of the solution as approximated by a finite volume method. The calculation of an accurate numerical flux  approximation $\fprec$ at cell boundaries of the coarse grid is based on numerical quadrature in time, i.e.
\[
	\fprec_{k+\frac 1 2} = \nquad_{t^n}^{t^{n+1}} g\of{v\left(x^-_{k+\frac 1 2}, \cdot\right), v\left(x^+_{k+\frac 1 2}, \cdot \right)} \approx  \int_{t^n}^{t^{n+1}} f(v(x, t))  \vd t.
\]
Our numerical tests used low order quadrature methods as we are especially interested in flux values for nonsmooth $u$ in space and time. Therefore high-order quadrature rules would be of little use. Our next problem consists of finding a suitable and well defined $\alpha_{k+\frac 1 2}$ that satisfies
\[
	\fnnnn_{\alpha_{k + \frac 1 2}} \approx \fprec_{k+\frac 1 2}.
\]
We therefore formulate the following definition
\[
	\fnnnn_{\alpha_{k + \frac 1 2}} = \min_{f \in \ch\of{h^n_{k+\frac 1 2}, g^n_{k+\frac 1 2}} }\norm{f - \fprec_{k+\frac 1 2}}_2 = \PR_{\ch\of{h^n_{k+\frac 1 2}, g^n_{k+\frac 1 2}} }\fprec
\]
of the target value of the neural network GT flux as the solution of a constrained optimization problem, i.e. the projection of the flux to the convex hull of the dissipative low order and non-dissipative high-order flux. This formulation is usable in single conservation laws as well as for systems of conservation laws and extends also to more general convex combinations if additional fluxes are added for the usage in convex combinations. A different norm or different convex functional could be used instead, investigation in such direction will be open for the future.
 An additional complexity stems from the fact while that the above minimization problem always has a unique solution as both the objective as also the domain is convex the situation is worse for the following related minimization problem
\[
	\alpha = \argmin_{\tilde \alpha \in [0, 1]} \norm{\fnn_{\tilde\alpha} - \fprec_{k+\frac 1 2}}_2.
\]
The solution is in fact not unique for $g = h$ which happens for example for $u = \mathrm{const.}$. We  make use of the following definition 
\[
	\alpha = \max\of{ \argmin_{\tilde \alpha \in [0, 1]} \norm{\fnn_{\tilde\alpha} - \fprec}_2}
\]
to select the most dissipative value of $\alpha$ in the degenerate case. The numerical solution of this problem in the case of the $2$-norm is based on the usage of the Penrose inverse 
\[
	b = \fprec - g, \quad A = h-g, \quad \beta = \min(1, \max(0, A^\dagger b)), \quad \alpha = 1- \beta,
\]
as a simple calculation shows. 
The affine-linear map \[M_{k+\frac 1 2}: \R \to \R^p, \beta \mapsto \beta h_{k+\frac 1 2} + (1-\beta) g_{k+\frac 1 2} = g_{k+\frac 1 2} + \alpha (h_{k+\frac 1 2} - g_{k+\frac 1 2}) = w + A\alpha\] can be expressed in the standard basis using the matrix $A_{k+\frac 1 2} = h_{k+\frac 1 2} - g_{k+\frac 1 2}$ and the support vector $w_{k+\frac 1 2} = g_{k+\frac 1 2}$. The value $\beta$ controls an affine combination, where $\beta = 1- \alpha$ yields the identical value as before using the blending scheme. One therefore finds 
\[
	\argmin \norm{w + A\beta - f^{n, precise}}_2 = \argmin \norm{A\beta - \underbrace{(f^{n, precise} - w)}_{b}}_2 = A^\dagger b
\]
for the projection of $f^{n, precise}$ onto the subspace $\ran{M}$.
As the Penrose inverse is not only the least squares, but also the least norm solution. It has also the smallest absolute value, i.e. $\beta = 0$, in the case that $A$ is degenerate. The distinction between $\alpha$ and $\beta$ was made to enforce $\alpha = 1$ for degenerate A. As we are interested in the projection of $f^{n, precise}$ onto $M_{k+\frac 1 2}([0, 1])$, one concludes that if the unconstrained minimizer lies outside of the image of $[0, 1]$ under $M$, the constrained minimizer must be one of the edges and in fact, the edge lying nearer to the unconstrained minimizer. This yields the given formula for the minimizer.

\section{Polynomial Annihilation Based Scheme} \label{se_regul}

In this chapter, we want to propose another possibility to approximate blending parameter $\alpha$. Therefore, the construction of polynomial annihilation operators in one spatial dimension is visited at first. Since these operators approximate the jump function of a given sensing variable, we use them in a second step to propose a choice of $\alpha$.

\subsection{Polynomial Annihilation-Basic Framework}

The general idea of polynomial annihilation (PA) operators proposed in \cite{archibald2005polynomial} is to approximate the jump function
\begin{equation}
 [s](x) = s(x^+)-s(x^-)
\end{equation}
for a given $s:\Omega\rightarrow\mathbb{R}$ referred to as the sensing variable. Therefore, we want to construct an operator $L_m[s](\xi)$ which gives an approximation of $[s](\xi)$ of $m$-th order. 

For a given $\xi\in\Omega$, we first choose a stencil of $m+1$ grid points around $\xi$
\begin{equation}
 S_\xi = (x_k,\ldots, x_{k+m})\hspace{1em}\text{with}\hspace{1em} x_k\leq \xi\leq x_{k+m}.
\end{equation}
The name giving polynomial annihilation is performed in the next step by defining the annihilation coefficients $c_j$ implicitly by
\begin{equation}
 \sum_{x_j\in S_\xi} c_j p_l(x_j) = p_l^{(m)}(\xi)
\end{equation}
where $\{p_l\}_{l=0}^m$ is any basis of the space of all polynomials with degree $\leq m$. Note that the coefficients $c_j$ only depend on the choice of the stencil $S_\xi$ since the $m$-th derivative $p_l^{(m)}$ of a polynomial with degree $\leq m$ is constant. Expanding this approach to higher spatial dimensions would cause the $c_j$ to directly depend on $\xi$.
Finally, we also need a normalization factor $q_m$ calculated by
\begin{equation}
 q_m = \sum_{x_j\in S_\xi^+} c_j,
\end{equation}
where $S_\xi^+ = \{x_j\in S_\xi| x_j\geq\xi\}$. The normalization factor $q_m$ is constant for fixed choice of $S_\xi$. We define the PA operator of order $m$ now by
\begin{equation}
 L_m[s](\xi) = \frac{1}{q_m} \sum_{x_j\in S_\xi} c_j s(x_j).
\end{equation}
In \cite{archibald2005polynomial} it was shown that
\begin{equation}
 L_m[s](\xi) =\begin{cases}
 			   [s](\tilde{x}) + \mathcal{O}\left(\tilde{h}(\xi) \right), &\text{if } x_{j-1}\leq\xi,\tilde{x}\leq x_j,\\
 			   \mathcal{O}\left( (\tilde{h}(\xi))^{\min(m,l) }\right) , & \text{if } s\in C^l([x_k,x_{k+m}]).
 			  \end{cases}
\end{equation}
Here, $\tilde{x}$ denotes a jump discontinuity of $s$ and $\tilde{h}(\xi):=\max\{|x_i-x_{i-1}|~|~x_i,x_{i-1}\in S_\xi \}$.

\subsection{Scheme based on polynomial annihilation}

Based on the above-presented framework, we now construct the convex parameter $\alpha$. Since entropy could be dissipated in cells where discontinuities occur, the parameter is supposed to be $0$ in regions where the solution is smooth and $1$ in discontinuity containing cells. This can be achieved by PA operators which are not constructed to give the location of a discontinuity but to approximate the height of the jump at that location. Hence, we need to normalize the operator by a factor approximating the height of a typical jump, i.e. $\frac{1}{z}L_{2p}[s]$ with a normalization factor $ z\approx[s](\tilde{x})$. Here, the PA operator is used on the four-point stencil $(x_{k-p+1},\ldots,x_{k+p})$ and the corresponding mean values $(u_{k-p+1}^n,\ldots,u_{k+p}^n)$ for  a given $n$. 
This normalization factor $z$ is also provided by a PA operator. Therefore, we apply $L_{2p}$ to the idealized values 
\begin{equation*}
 (u_{\max}^n,\ldots,u_{\max}^n,u_{\min}^n,\ldots,u_{\min}^n)
\end{equation*}
based on the same four-point stencil where 
\begin{equation*}
 u_{\max}^n = \max\{u_{k-p+1}^n,\ldots,u_{k+p}^n\},\hspace{1em} u_{\min}^n = \min\{u_{k-p+1}^n,\ldots,u_{k+p}^n\}.
\end{equation*}
With this normalization factor, the natural selection of $\alpha$ would be
\begin{equation*}
 \alpha = \frac{L_{2p}[u]}{z}.
\end{equation*}
This choice does not fulfill the before mentioned recommended property since the normalization gives a much more accurate approximation of the jump height. Actually, the jump function is approximated flatter in case of using $L_{2p}[u]$ on $(u_{k-p+1}^n,\ldots,u_{k+p}^n)$. Another occurring problem is the normalization factor $z$ equal to zero. That is whenever $u_{\max}^n= u_{\min}^n$ holds. A solution can be obtained by a simple regularization. Considering both these issues, we choose
\begin{equation}
 \alpha^n = \frac{c_1 L_{2p}[u]}{z + c_2},
\end{equation} 
where $c_2>0$. Experiments showed that $c_1=10$ is an appropriate choice to compensate the difference between the accuracies of the approximations. The regularization is picked as $c_2 = \|\bu\|_1$, where
\begin{equation*}
 \|\bu^n\|_1 = \sum_{i=1}^N \frac{|u_i^n|}{N}\mu(\Omega)
\end{equation*}
is the discrete $L^1$-norm. The order of the applied PA operators is selected as $p=4$ in this work. In a last step, we apply a sup-mollification to define the predictor
\begin{equation}
 \tilde{\alpha}^n = \alpha^n\circledast \max\left \{1-\frac{1}{3}\norm{\frac{x}{\Delta x}},0 \right\}.
\end{equation}

\section{Numerical Experiments}\label{se_numerics}
In the following part, we investigate and compare the approaches described before in our blending schemes \ref{def_Blending}.  Especially, we
focus on the following questions: 
\begin{itemize}
	\item Are all schemes  able to capture shocks and are they oscillation free?
	\item Does the discrete entropy inequality \ref{eq:eie} also hold for the PA and data-driven selection of $\alpha$? Are the schemes violating other physical constraints like the positivity of density and pressure?
	\item Which order of convergence can we expect from our schemes?
\end{itemize}
We test our schemes on the Euler equations of gas dynamics  \eqref{eq:Euler}  and use the mathematical entropy   \eqref{eq:EulerPE}. To test the high-order accuracy of the proposed schemes, we consider a smooth connected density variation from  \cite{klein2021using}. To test and compare our schemes for problems incorporating
strong shocks, we focus on the famous Shu-Osher problem (Test Number 6 from  \cite{SO1988, SO1989}).
The Shu-Osher test case is nowadays working as  a benchmark problem in computational fluid dynamics. 
We compare four different schemes with each other and also tested the Dafermos criterion based scheme from \cite{klein2021using}. We consider in detail: %
\begin{enumerate}
	\item The discrete local entropy stable scheme with $\alpha$ dictated by \textbf{Condition F} using a second-order entropy conservative flux with SSPRK(2, 2) time integration and the local Lax-Friedrichs flux (LLF) with forward Euler method. We denote the scheme in the following with DELFT.
	\item A positivity preserving scheme with $\alpha = \max(\alpha_\rho, \alpha_p)$ determined by \textbf{Condition $\rho$} and \textbf{Condition $P$}. This scheme is using the fourth-order entropy conservative flux and  the LLF flux with SSPRK(3, 3) time integration. The scheme is called PPLFT. 
	\item A data-driven scheme, $\alpha$ is determined using  neuronal nets. The high-order flux is again fourth-order entropy conservative combined with an SSPRK(3,3) while the low order entropy stable flux is given by the  LLF flux together with an SSPRK(2,2) method in time.
The scheme is denoted with DDLFT. 
	\item A polynomial annihilation based scheme where $\alpha$ is determined through the technique described in \cref{se_regul}. The high-order flux is the fourth-order entropy conservative flux with SSPRK(3,3) quadrature while the low order entropy stable flux is given by the Lax-Friedrichs scheme using the SSPRK(2,2) quadrature in time. The PA operators used are of fourth-order. We denote the scheme with PALFT.
	\item The Dafermos criterion based method developed and proposed in \cite{klein2021using}.
\end{enumerate}
The different decisions concerning the fluxes and their time quadrature are rooted in the following observations. The stencil of the fourth-order in space, SSPRK(3,3) in time flux is 12  cells wide, whereas the second-order in space and second-order in time flux has only  4 points wide stencils. Therefore, a discontinuity results in bigger values for $\alpha$ using \textbf{Condition $F$} if the base stencil is wider because even a distant discontinuity impacts the selection of $\alpha$. As a small $\alpha$ implies a more accurate method, we therefore use the smaller stencils for the provably discretely entropy dissipative method.
To avoid the decrease in accuracy for the high-order methods even in smooth regions, one can use subcell techniques \cite{Harten1989ENO, rueda2022subcell} instead.  \\
The PPLFT method uses uniform SSPRK time integration on the other hand as we expect no gains from splitting the time integration in this case. The recalculation of $\alpha$ was therefore carried out in every sub step of the RK method for PPLFT method, whereas it was carried out only once in every time step for the last two methods (PALFT and DDLFT). The splitting of the time integration therefore allows us to make use of a speed improvement by only calling the neuronal network once every time step. A second reason to use the quadrature based definition of high-order time integration lies in the fact that the $\alpha$ used as training data for the network is the least-squares solution described in section \ref{se_neunet}. A more accurate flux allows the value of $\alpha$ to be shifted towards zero in smooth regions in the training, therefore lowering the numerical dissipation. The same combination of fluxes are also used for the last scheme for comparison reasons. Before comparing our schemes, we explain how we generate our training data for DDLFT.
	\subsection{Calculation of Training Data}
		We used initial data
	\[
		\begin{pmatrix}\tilde \rho_0(x, t)\\ \tilde v_0(x, t) \\ \tilde p_0(x, t) \end{pmatrix} = \sum_{k=1}^{N} a_k \sin(2\pi k x) + b_k\cos(2 \pi k x)
	\]
	with $N = 100$ and the vectors $a, b \in \R^{3 \times N}$ were selected by selecting random $\tilde a, \tilde b \in [0, 1]^{3 \times N}$ and afterwards scaling them by $
		\hat a_k = \frac {\tilde a_k}{k^{c}}$ and $ \hat b_k = \frac{\tilde b_k}{k^c}.$
	This scaling is based on the fact that the Fourier coefficients $a_k, b_k$ of functions $f \in C^l$ satisfy $a_k, b_k \in \mathcal O(1/k^l)$ and we used therefore $c =2.0$ to generate random functions with suitable regularity. We would like to ensure positive pressure and density with a suitable amplitude to give a representative picture of the dynamics of the system but without an unsuitably small CFL restriction on the time step. Therefore, the following modifications were done 
	\[
		\begin{aligned}
			 a_{1, k} &= (0.2 + A_\rho)\hat a_{1, k}, \quad  \hat  a_{2, k} = A_v a_{2, k},\quad  \hat a_{3, k} = (0.3 + A_p), \\
			\rho_0(x, t) & = \tilde \rho_0(x, t) - \min_x(0, \tilde \rho_0(x, t)), \quad  v_0(x, t) = \tilde v_0(x, t) + B_v, \\  p_0(x, t)  &= \tilde p_0(x, t) - \min_x(p_0(x, t), 0)
		\end{aligned}
	\]
	to ensure positivity. The values $A_\rho, A_v, B_v, A_p$ were randomly selected to satisfy $A_\rho \in [0, 2], A_v \in [0, 2], B_v \in [-2, 2], A_p \in [0, 4]$.  
	The generated initial data was afterwards solved  by a second-order MUSCL scheme and SSPRK(3,3)  with $4000$ cells. The MUSCL scheme was selected because it is a highly robust method. In total $32$ different initial conditions were generated, solved up to $t = 1$ and sampled at every cell interface of the coarse grid at $100$ equally spaced times in the available interval.
	\subsubsection*{Layout and Training of the Network}
	\begin{table}
		\begin{tabular}{ccccccc}
			Layer 				& input & 2 	& 3 	& 4 	& 5 	& output\\
			Activation 			& ELU 	& ELU 	& ELU 	& ELU 	& ELU 	& $x \to x$\\
			Number of Neurons 	& 40 	& 80 	& 80 	& 80 	& 80 	& 1\\
		\end{tabular}
	\caption{Used network structure}
	\label{tab:net}
	\end{table}
	We use a neural  network built out of six layers whose dimensions are given in table \ref{tab:net}. In all, but the last layer, the $ELU$ activation function is applied. The inputs are  the values of the conserved variables and the pressure of five cells left and right to the cell boundary where $\alpha$ has to be determined.
	Our network for the prediction of $\alpha$ was trained using the ADAM optimizer \cite{kingma2017adam} with parameters scheduled as given in  \cref{tab:sched}.
		The resulting loss curve is printed in \cref{fig:loss}. The training took circa 20 minutes on 8 cores of a AMD Ryzen Threadripper 5900X at 3.7 Ghz. 
	\begin{table}
		\begin{center}
		\begin{tabular}{cc cc cc cc}
			Section 	& 1 	& 2 	& 3 	& 4 	& 5 	& 6		& 7 \\
			Epochs 		& 25 	& 25 	& 25 	& 25 	& 25 	& 25 	& 25\\
			Batchsize 	& 32 	& 256 	& 1024 	& 4096 	& 4096 	& 4096  & 4096\\
			Stepsize  	&0.001 	& 0.001 & 0.001	& 0.001	& 0.0001& 0.00001& 0.000001\\
		\end{tabular}
		\caption{Overview of used training parameters}
		\label{tab:sched}
		\end{center}
	\end{table}	
	\begin{figure}
	\begin{center}
	\includegraphics[width=0.5\textwidth]{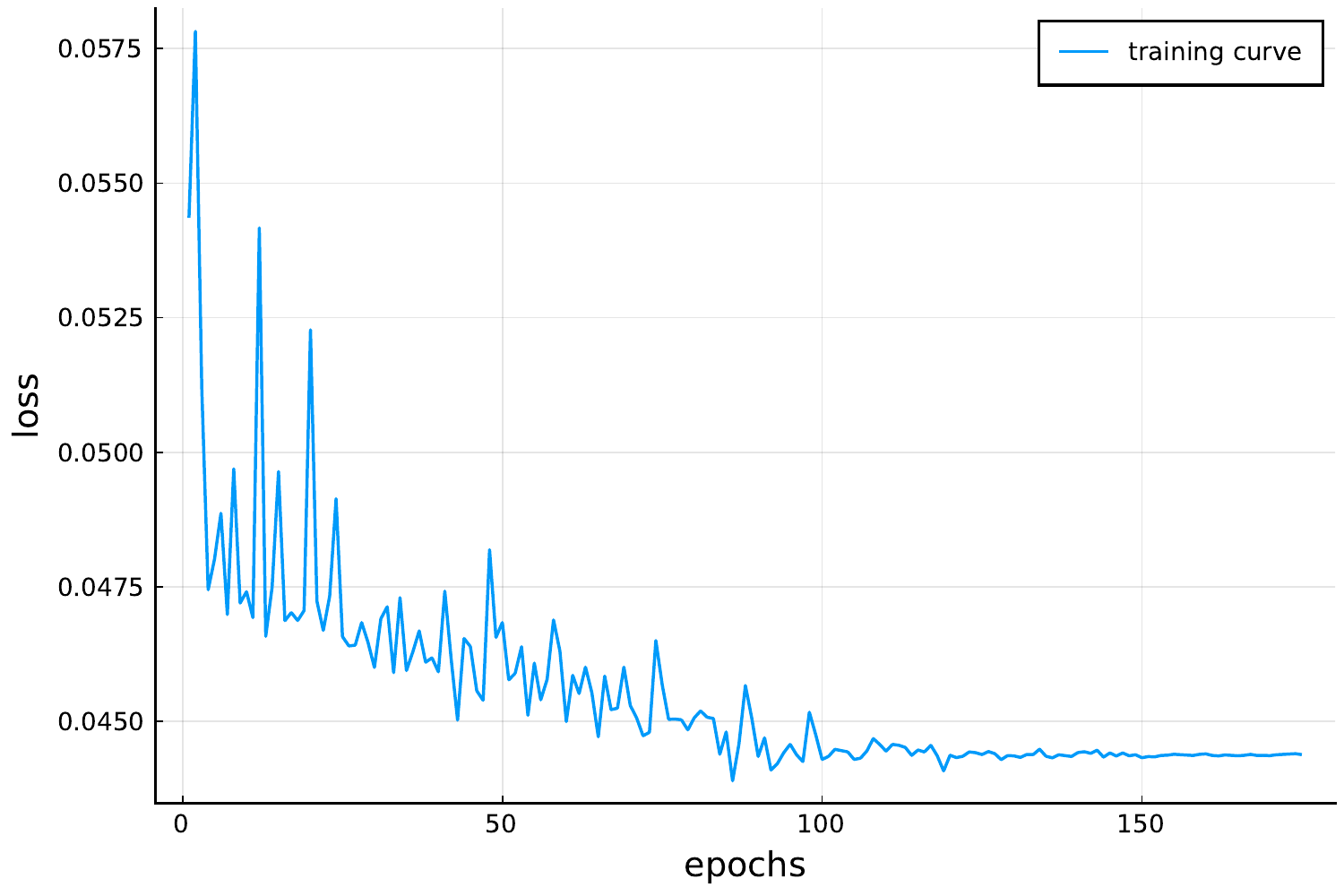}
		\caption{Training loss of the neuronal network.}
		\label{fig:loss}
		\end{center}
		\end{figure}
	\subsection{Testing of the Schemes}
	\begin{figure}
		\centering
		\begin{subfigure}{0.48\textwidth}
			\includegraphics[width=\textwidth]{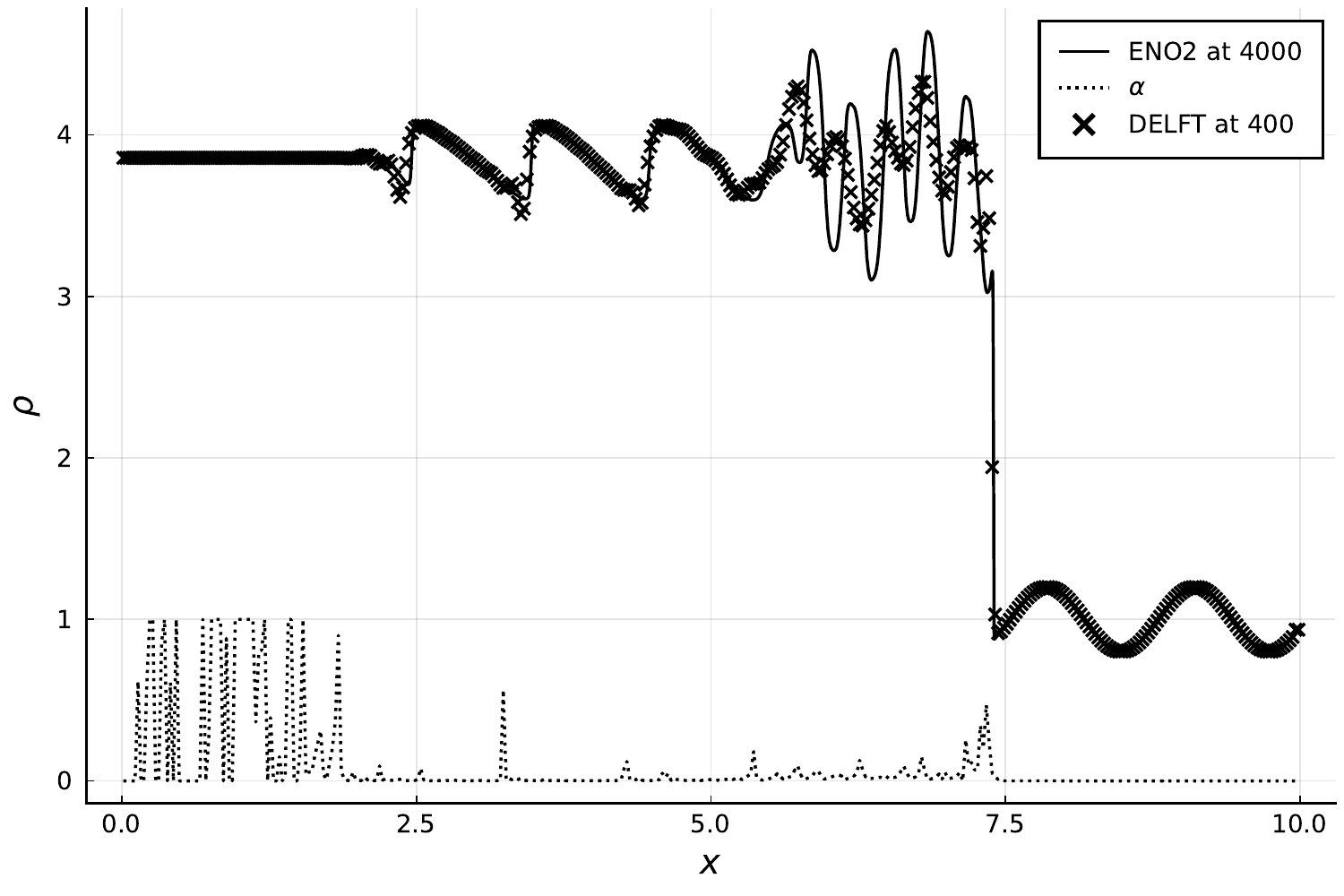}
		\end{subfigure}
		\begin{subfigure}{0.48\textwidth}
			\includegraphics[width=\textwidth]{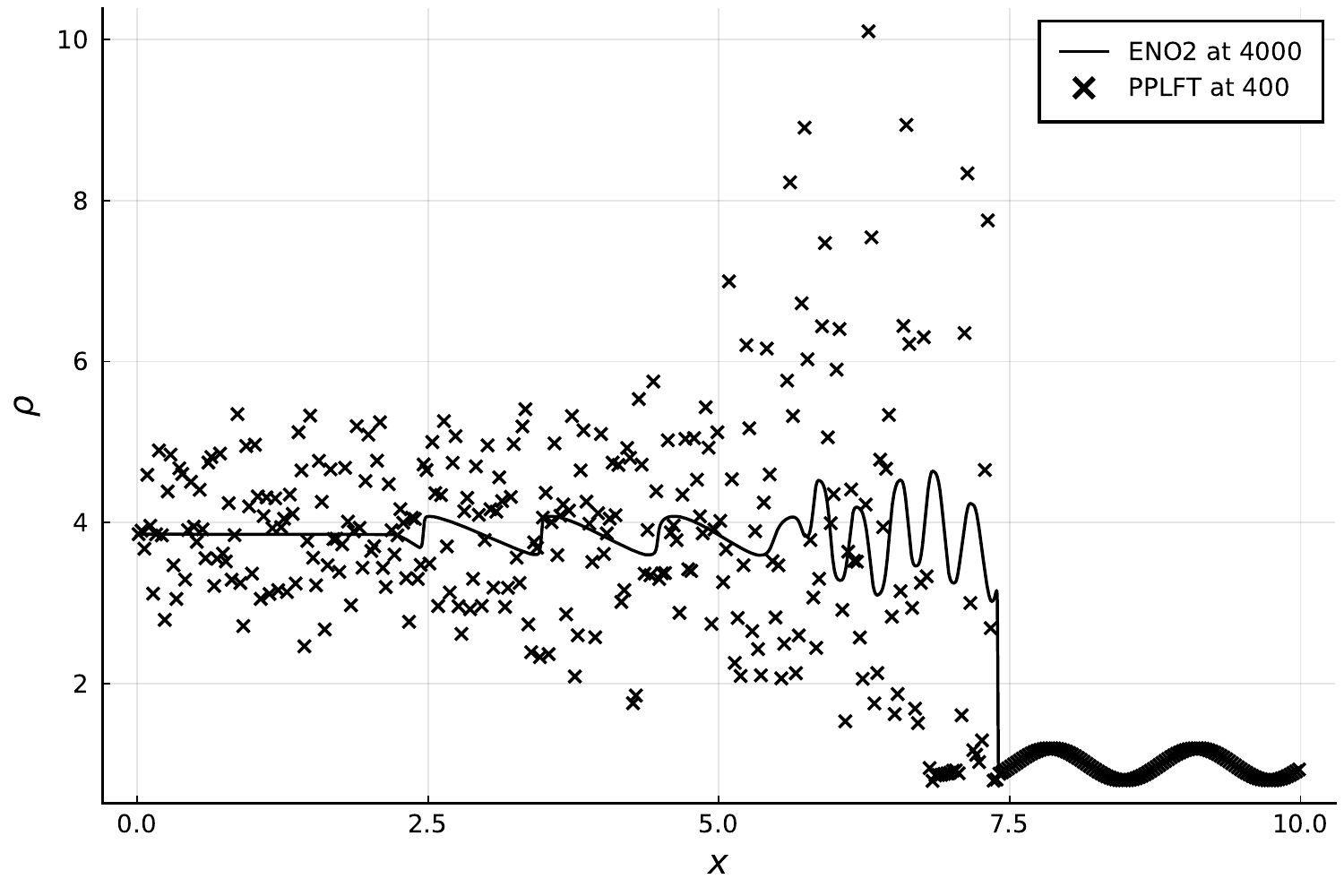}
		\end{subfigure}
		\begin{subfigure}{0.48\textwidth}
			\includegraphics[width=\textwidth]{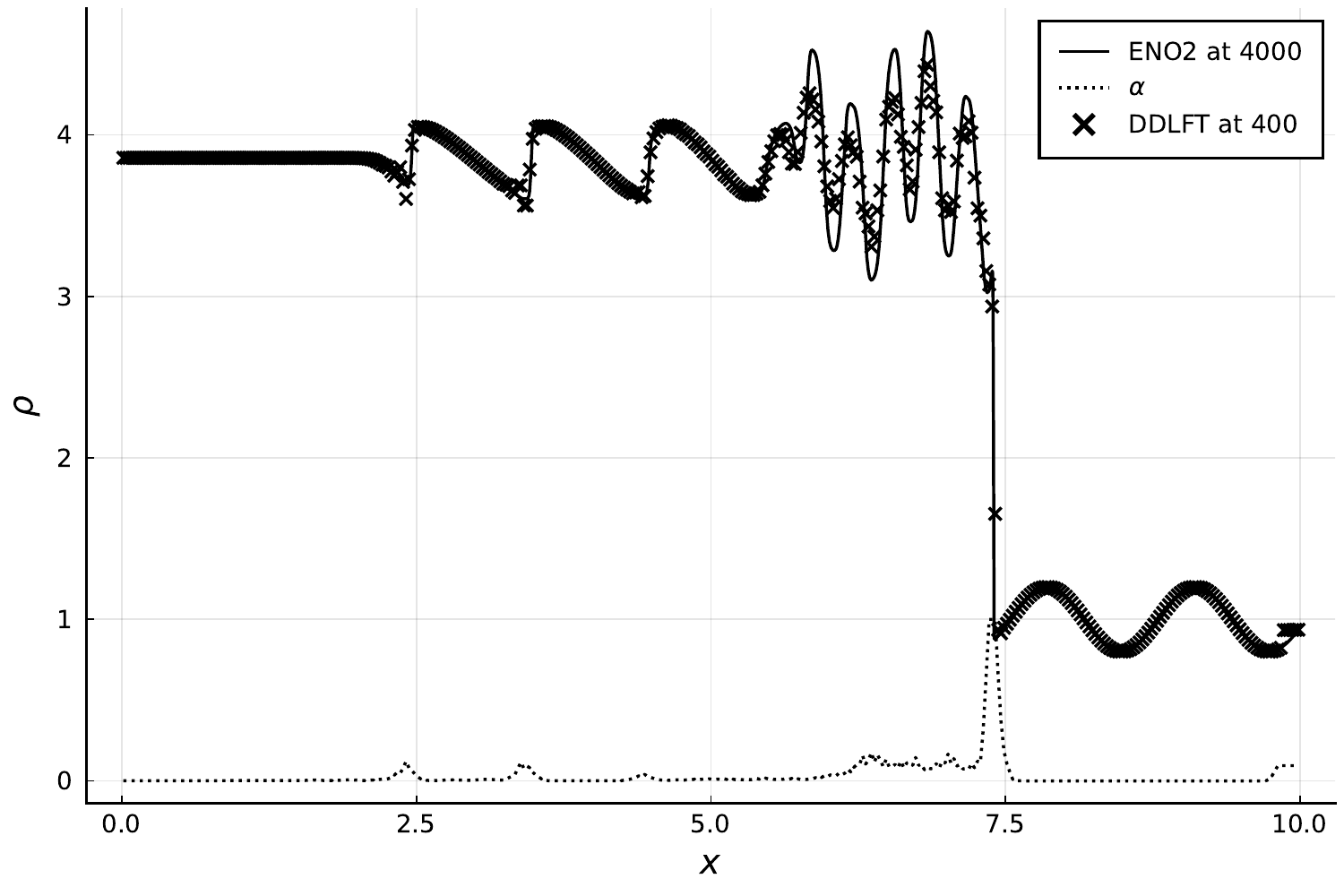}
		\end{subfigure}
		\begin{subfigure}{0.48\textwidth}
			\includegraphics[width=\textwidth]{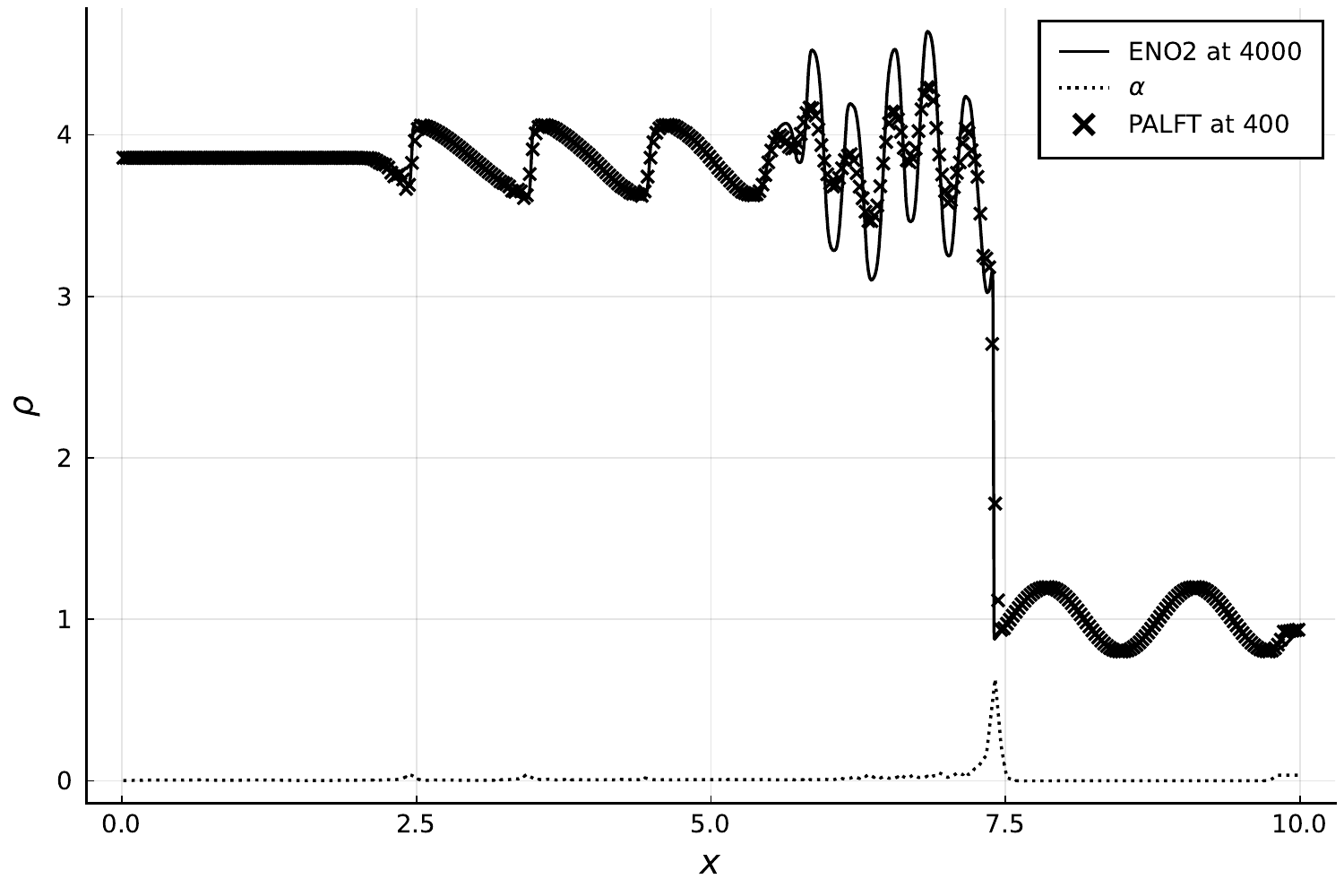}
		\end{subfigure}
		\begin{subfigure}{0.48\textwidth}
			\includegraphics[width=\textwidth]{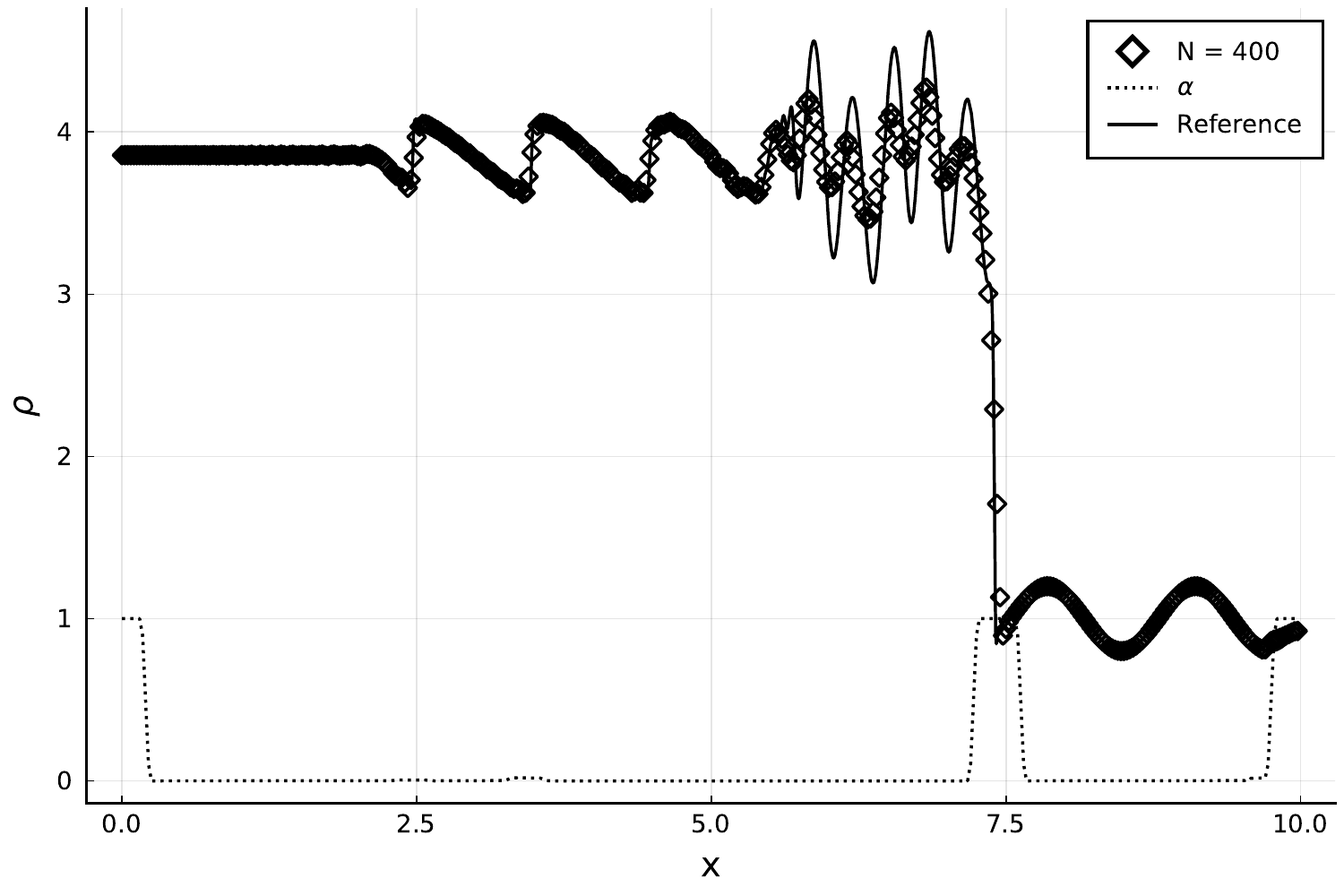}
			\subcaption{Dafermos Criterion scheme from \cite{klein2021using}}
		\end{subfigure}
		\caption{Density profile  $T = 1.8$}
		\label{fig:SO6}
	\end{figure}
The initial conditions of the Shu-Osther test are given by
	\begin{align*}
		\rho_0(x, 0) = \begin{cases}3.857153  \\ 1 + \epsilon \sin(5 x)  \end{cases} 
		\quad 
		v_0(x, 0) = \begin{cases} 2.629  \\ 0  \end{cases}
		p_0(x, 0) = \begin{cases} 10.333 & x < 1 \\ 1 & x \geq 1 \end{cases}
	\end{align*}
	in the domain $\Omega = [0, 10]$. The parameter $\epsilon = 0.2$ was set to the canonical  value of $0.2$ and the adiabatic exponent was set to $\gamma=\frac 7 5$ for an ideal gas. The density profiles  are printed in \cref{fig:SO6}. The numerical solutions are describing in nearly all cases the reference solutions   except for the positivity preserving scheme (PPLFT). Here, the calculated solution is obviously meaningless even if the positivity of the solutions is still ensured and the calculation could be carried on up to $T= 1.8$. The other three schemes  are able to resolve the strong shock without nonphysical oscillations. The amount of points needed for the transition is small and the wave structure trailing the shock is resolved accurate. The high gradient areas at $x=2.5, 3.5, 4.5$ results in small oscillations for the data-driven scheme   and PA based scheme. However, these oscillations are nearly not visible, especially for PALFT.  The dotted lines give also the $\alpha$ coefficients in the convex combination of our blending schemes. We  realize that for PALFT 
the $\alpha$s distinguish essentially from zero around the shock where for $DDLFT$ the lower-order method is also activated in smooth regions (i.e. $\alpha>0$).
	
	\begin{figure}
		\begin{subfigure}{0.32\textwidth}
			\includegraphics[width=\textwidth]{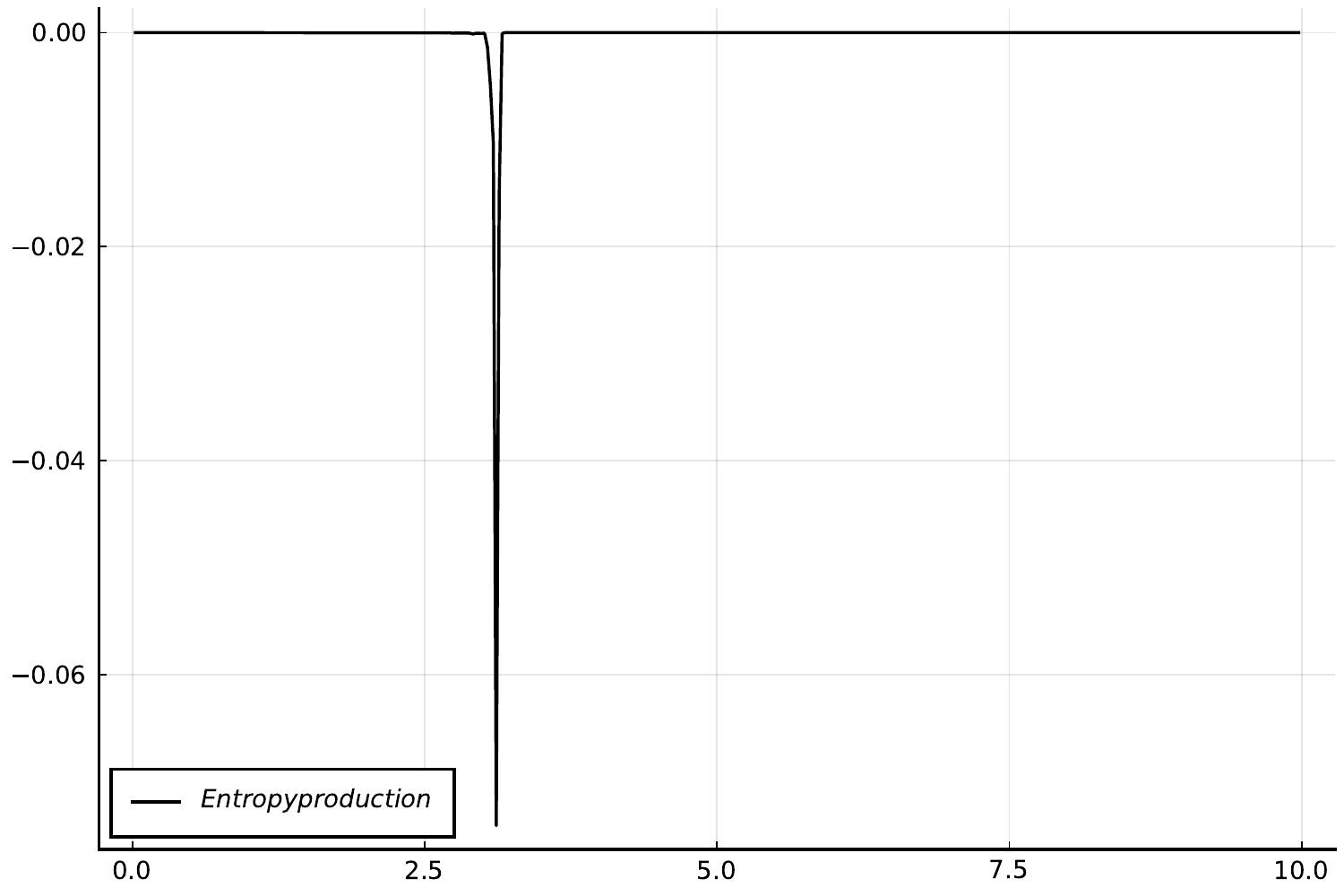}
			\caption{DDLFT}
		\end{subfigure}
		\begin{subfigure}{0.32\textwidth}
			\includegraphics[width=\textwidth]{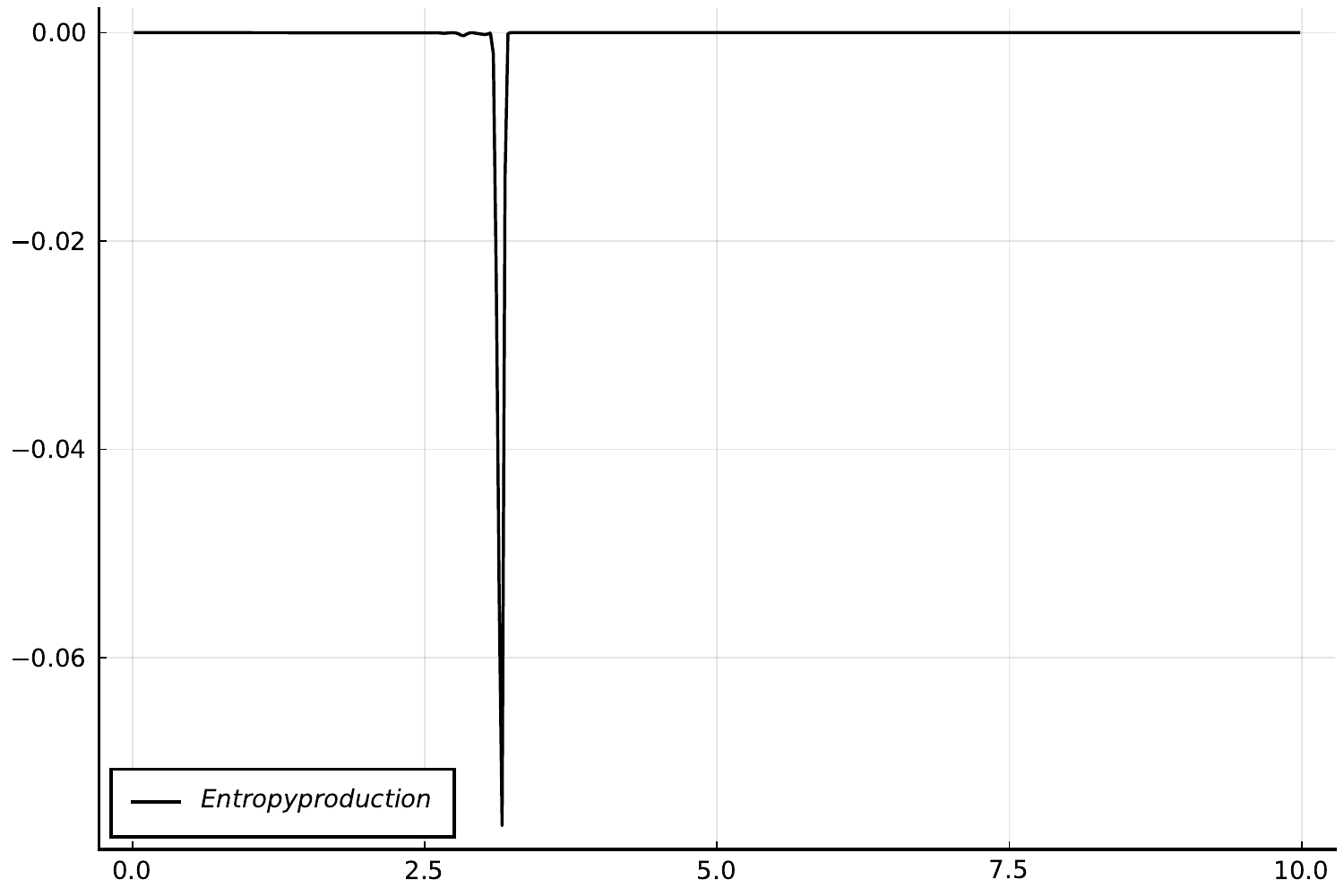}
			\caption{PALFT}
		\end{subfigure}
		\begin{subfigure}{0.32\textwidth}
			\includegraphics[width=\textwidth]{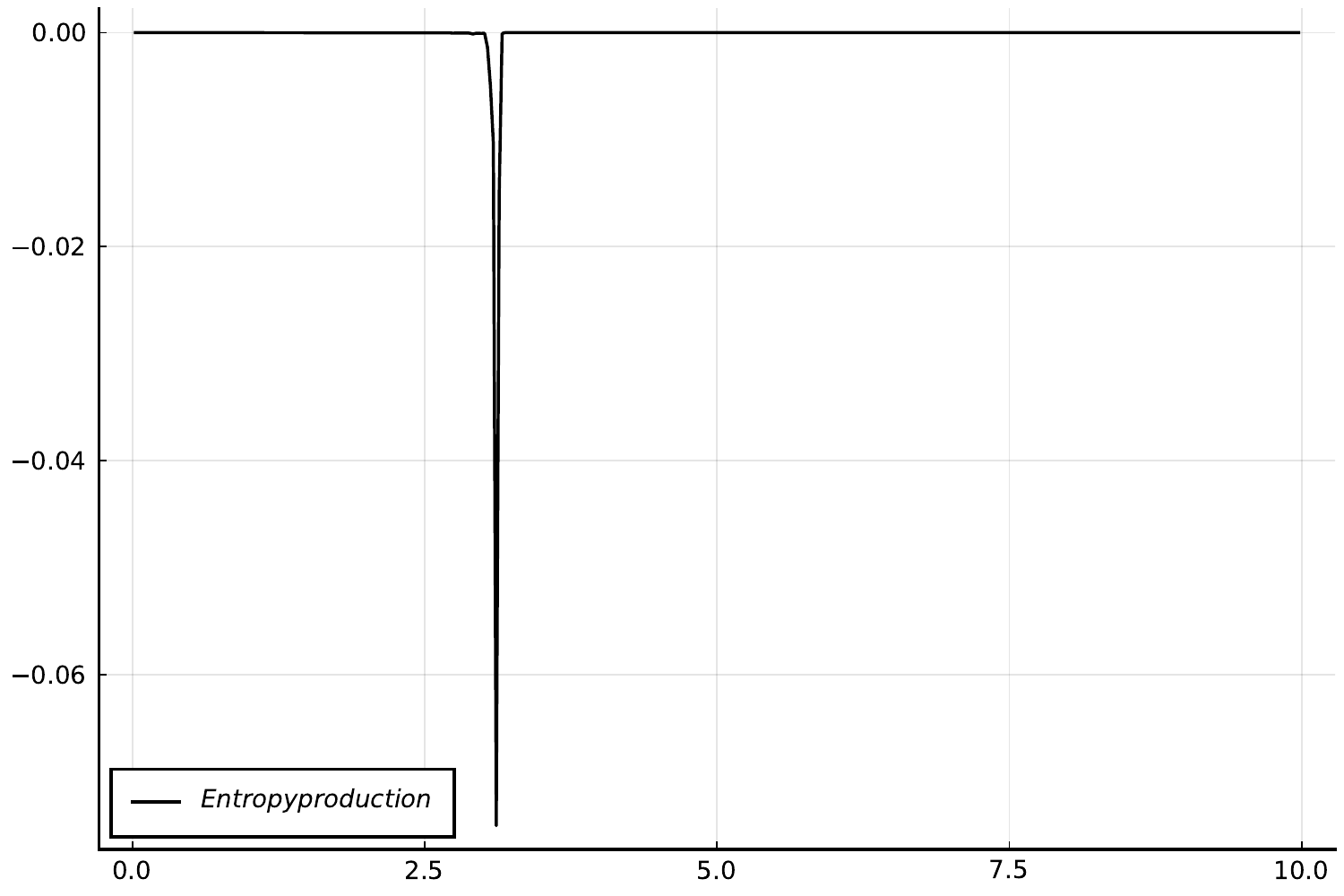}
			\caption{DELFT}
		\end{subfigure}
		\begin{subfigure}{0.32\textwidth}
			\includegraphics[width=\textwidth]{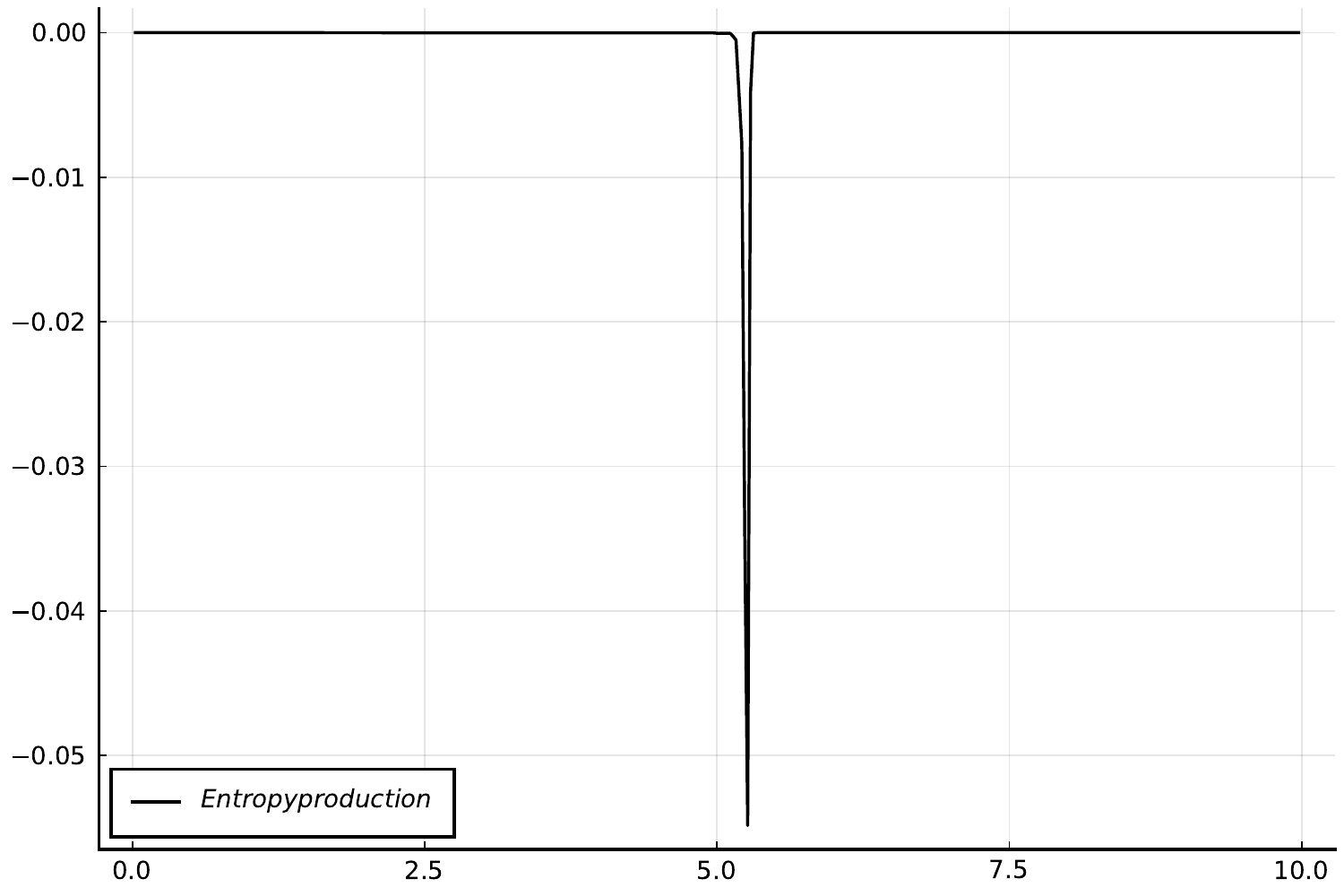}
			\caption{DDLFT}
		\end{subfigure}
		\begin{subfigure}{0.32\textwidth}
			\includegraphics[width=\textwidth]{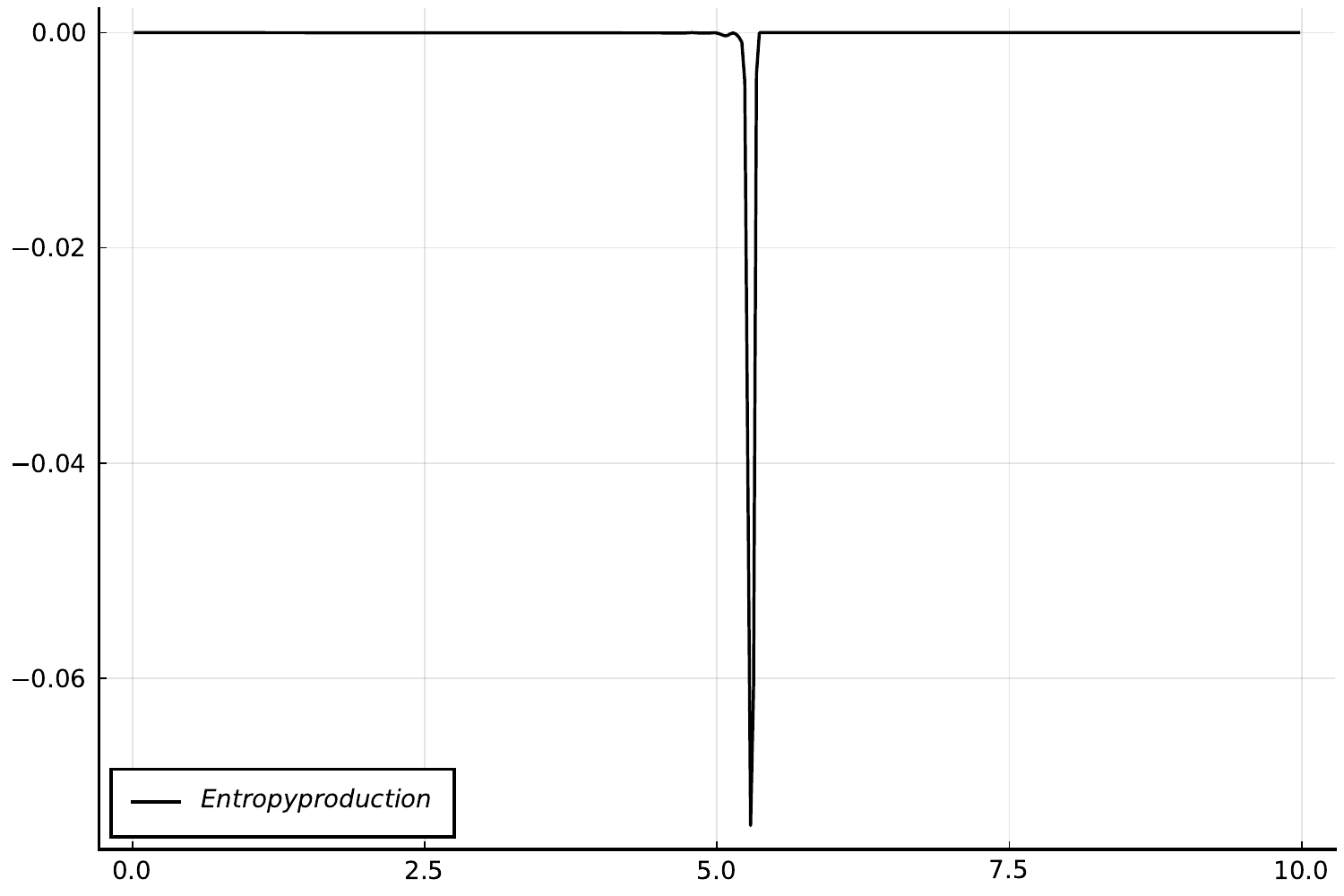}
			\caption{PALFT}
		\end{subfigure}
		\begin{subfigure}{0.32\textwidth}
			\includegraphics[width=\textwidth]{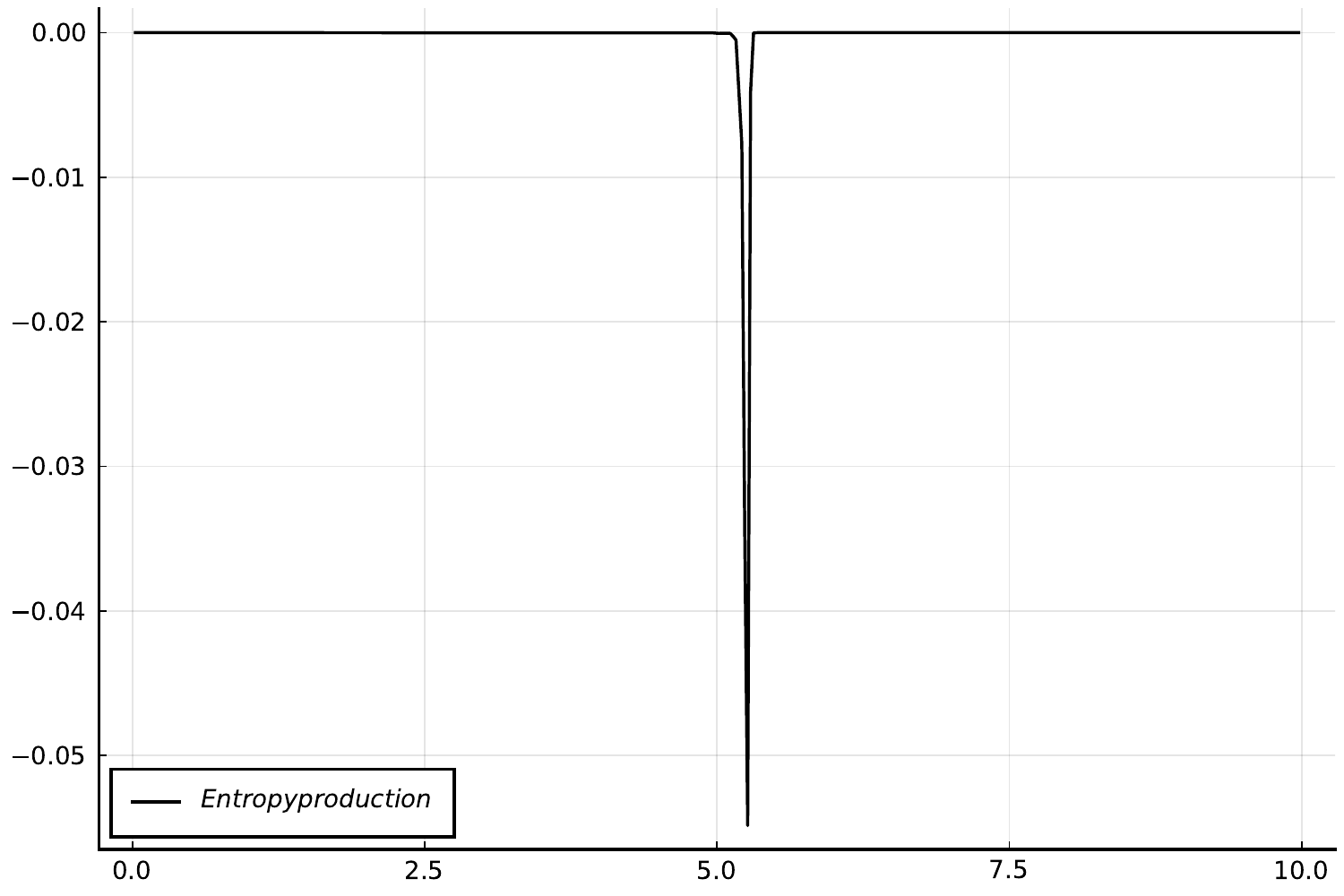}
			\caption{DELFT}
		\end{subfigure}
		\begin{subfigure}{0.32\textwidth}
			\includegraphics[width=\textwidth]{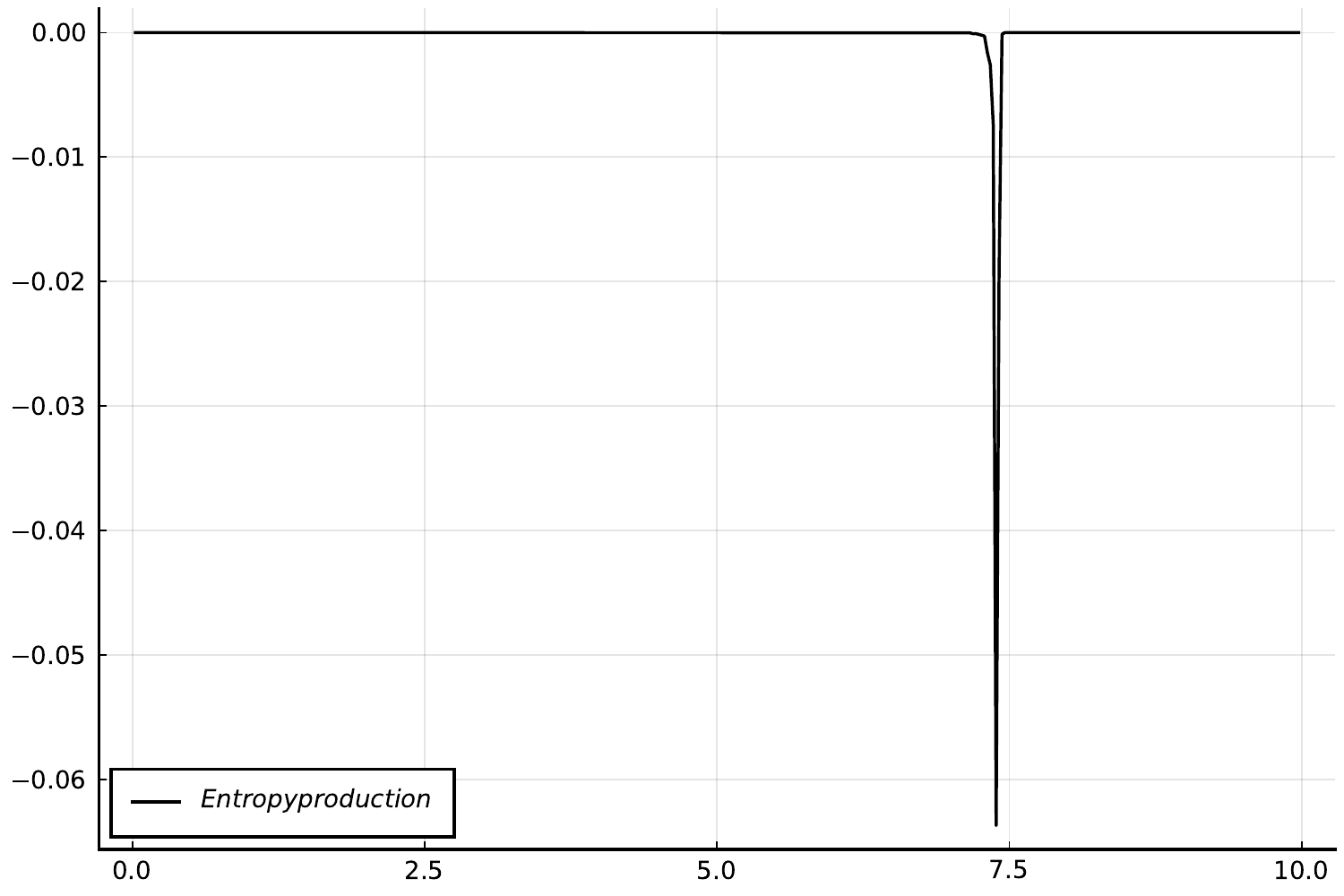}
			\caption{DDLFT}
		\end{subfigure}
		\begin{subfigure}{0.32\textwidth}
			\includegraphics[width=\textwidth]{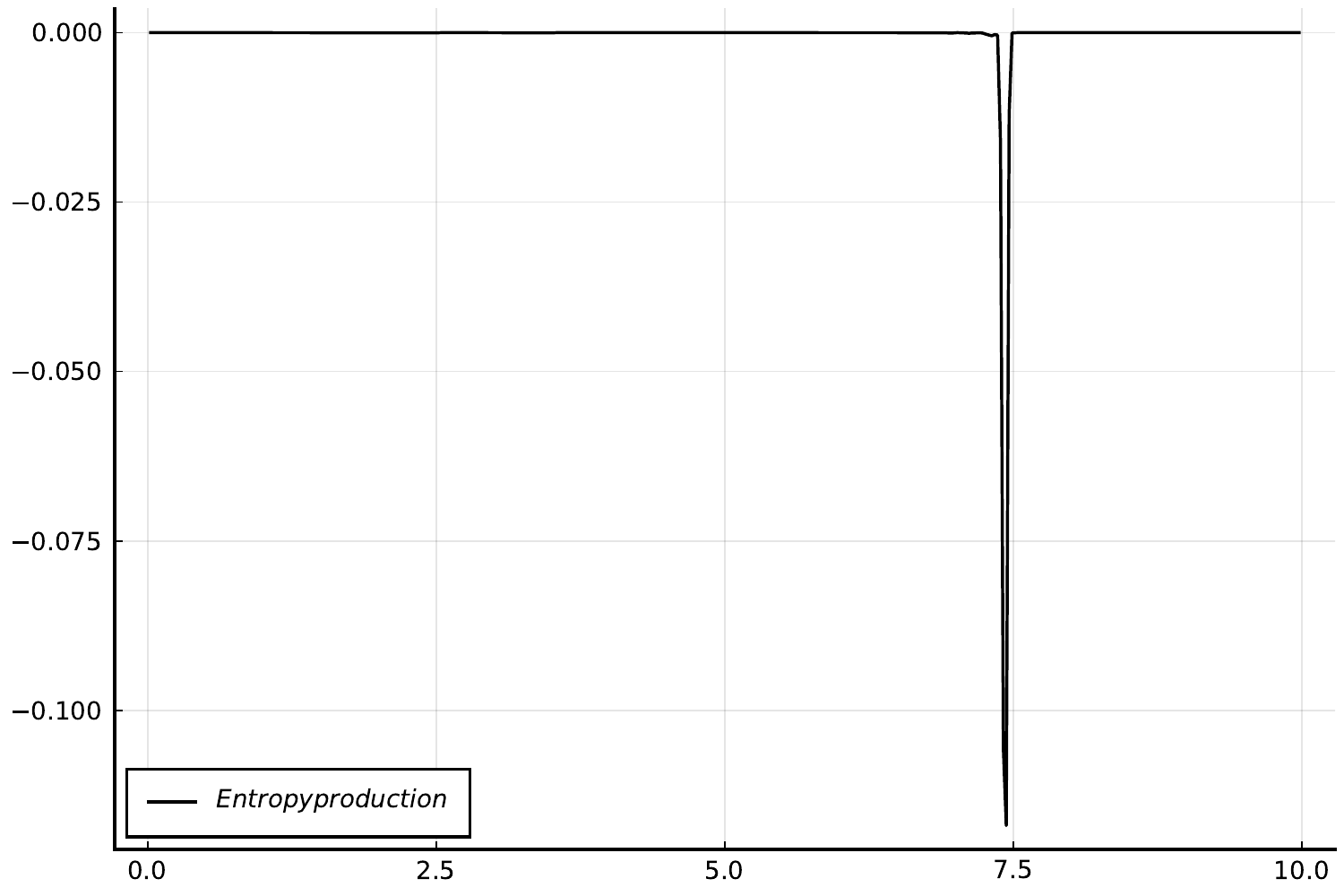}
			\caption{PALFT}
		\end{subfigure}
		\begin{subfigure}{0.32\textwidth}
			\includegraphics[width=\textwidth]{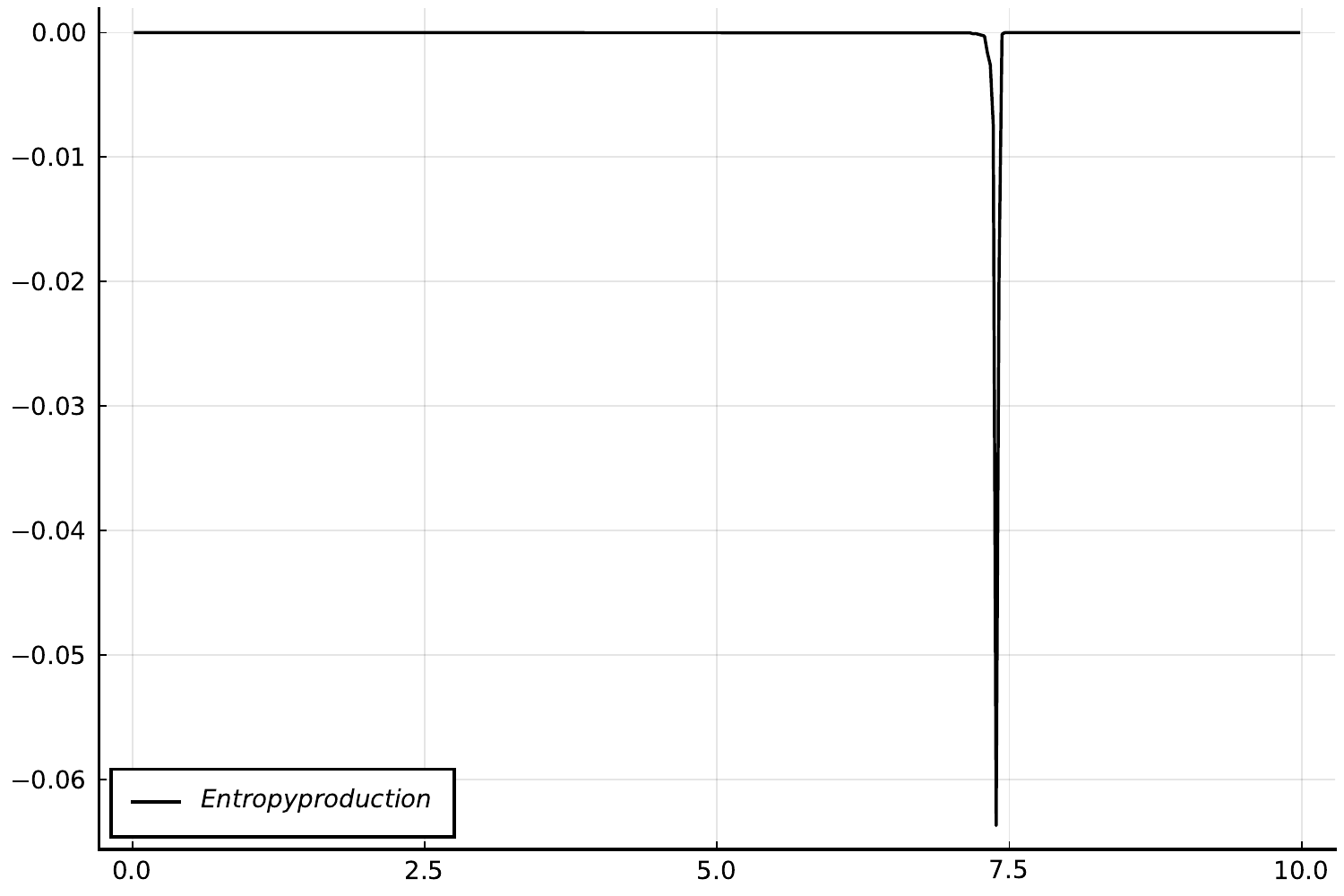}
			\caption{DELFT}
		\end{subfigure}
		
		\caption{Entropy production at $T = 0.6$ (first row), $T = 1.2$ (second row), $T = 1.8$ (third row)}
		\label{fig:SO7}
	\end{figure}

	 \cref{fig:SO7} shows the discrete entropy productions over the cells for the schemes 
	  DDLFT (left), PALFT(central) and DELFT(right) as snapshots at $T \in \{0.6, 1.2, 1.8\}$. 
	As printed, the schemes fulfill also locally the entropy inequality and are entropy dissipative (at least for this experiment). 
The small oscillations inside the numerical solution may be further cancelled out using additionally the Dafermos criterion \cref{eq_Dafermos} as mentioned also in \cite{klein2021using}.
Finally, we stress out that we have in all of our  simulations  no violations of positivity of density and pressure recognized. 
	\begin{figure}
	\begin{center}
		\includegraphics[width=0.7\textwidth]{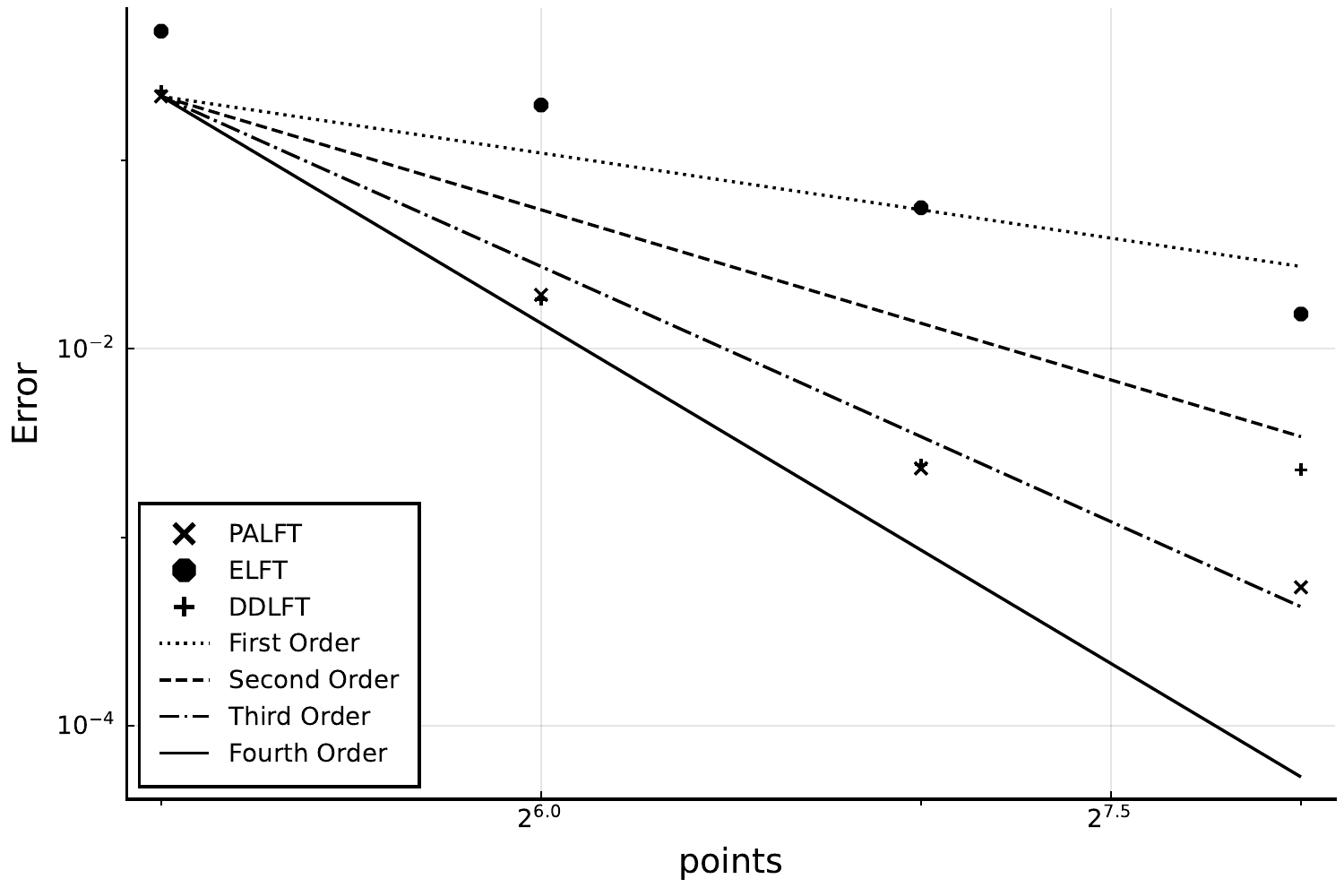}
		\caption{Convergence analysis of the schemes.}
		\label{fig:CA}
		\end{center}
	\end{figure}
	
	To determine the experimental order of convergence, we simulate the smooth transport of a density variation under pressure equilibrium already used in \cite{klein2021using}:
	\[
	\rho_0(x, 0) = 3.857153 + \epsilon \sin(2 x),  \quad v_0(x, 0) =  2.0, \quad p_0(x, 0) = 10.33333.
	\]
	The $\mathrm{L}^1$-errors of the schemes are shown in  \cref{fig:CA}. The reference was calculated by the ENO2 method  on a grid with $2^{14}$ cells. Please note that we calculated for a given grid $T_N$ with $2^N$ cells  the errors between the mean values of our approximated solution and the mean values of the reference solution using $2^{14}$ cells. 
	The values $\tilde u_k$ of a solution with $2^N$ cells is scaled  downwards to a grid with $2^{N-1}$ cells by the following procedure. We use the solution $\tilde u_k$ with $2^N$ cells and apply $ u_k = \frac{\tilde u_{2k} + \tilde u_{2k-1}}{2}$ to find the solution $u_k$ on the grid with $2^{N-1}$ cells. This procedure can be applied several times to find consistent mean values for any grid having a power of two cells. 
	The PA and data-driven schemes converge with third-order while the provable entropy stable scheme only converges with second-order convergence rate. Higher orders of convergence by using higher-order time integration and entropy conservative fluxes could not be demonstrated for schemes based on \textbf{Condition $F$} since  \textbf{Condition $F$} gives to big values for $\alpha$ to reach third-order convergence. The reason for this is that condition $F$ is only a sufficient condition for a satisfied fully discrete entropy inequality. A fully discrete entropy inequality can be also satisfied by smaller values of $\alpha$ as could be seen in \cref{fig:SO7}. A deeper analysis of sharper lower bounds on $\alpha$ will be part of future publications considering also subcell techniques. We further see a slide decrease of order for the DDLFT for fine grids. This is due to the fact that the training  data and the neural nets can not keep up with the DDLFT scheme itself on fine grids. 
	Finally,  we mark that third-order accuracy is only reached due to the time-integration method. 

\section{Summary}\label{se_summary}

In this work, we compared three different ways to control oscillations in a high-order finite volume scheme. After giving an introduction and an overview over the underlying numerical flux based on a convex combination, some physical constraints were concerned. To be more specific, we gave conditions that will assure the fully discrete entropy inequality or  pressure and density of the numerical solution of the Euler equations to be positive. A second possibility was further constructed using a feedforward neural network. Here, the network was trained by data which were calculated by a reference scheme. We provided afterwards a choice of the convex parameter based on polynomial annihilation operators after giving a brief introduction to their basic framework. In a last step, the resulting schemes were tested and compared by numerical experiments on the Euler equations. 
We have recognized that  our FV schemes except the PPLFT  were able to handle 
strong shocks and are mainly oscillation free. In respect to oscillations, we recognized the best performance for the DELFT scheme which is not suprising since it ensures the fully discrete entropy inequality. The drawback of this scheme was that only second order could be reached in our tests due to the selection of $\alpha$. However, subcell techniques can be used to solve this issue and will be investigated in the future. 
As we also recognized in our simulations, the selection of $\alpha$ using \textbf{Condition F} is sufficient but not necessary. By analyzing the remaining schemes, we have not recognized any violation of the entropy inequality even not for the data driven scheme.  We like to point out again that all of our considered approaches show promising results and can be used. \\
In the future, we plan to continue our investigation and consider two-dimensional problems using unstructured grids. Here, additional techniques are needed and we will also consider more advanced benchmark problems.  Extensions to multiphase flows are as well planned. 
Finally, our high-order FV blending schemes can be also the starting point of a convergence analysis for the Euler equations via dissipative measure-valued solutions  \cite{feireisl2019convergence, lukavcova2022convergence} which is already work in progress.


\bibliographystyle{siamplain}
\bibliography{literature}

\begin{thebibliography}{10}

\bibitem{abgrall2018general}
{\sc R.~Abgrall}, {\em A general framework to construct schemes satisfying
  additional conservation relations. application to entropy conservative and
  entropy dissipative schemes}, Journal of Computational Physics, 372 (2018),
  pp.~640--666.

\bibitem{abgrall2021relaxation}
{\sc R.~Abgrall, E.~L. M{\'e}l{\'e}do, P.~{\"O}ffner, and D.~Torlo}, {\em
  Relaxation deferred correction methods and their applications to residual
  distribution schemes}, arXiv preprint arXiv:2106.05005,  (2021).

\bibitem{abgrall2021analysis}
{\sc R.~Abgrall, J.~Nordstr{\"o}m, P.~{\"O}ffner, and S.~Tokareva}, {\em
  Analysis of the {SBP-SAT} stabilization for finite element methods part ii:
  Entropy stability}, Communications on Applied Mathematics and Computation,
  (2021), pp.~1--23.

\bibitem{abgrall2022reinterpretation}
{\sc R.~Abgrall, P.~{\"O}ffner, and H.~Ranocha}, {\em Reinterpretation and
  extension of entropy correction terms for residual distribution and
  discontinuous galerkin schemes: Application to structure preserving
  discretization}, Journal of Computational Physics,  (2022), p.~110955.

\bibitem{abgrall2017handbook}
{\sc R.~Abgrall and C.-W. Shu}, {\em Handbook of numerical methods for
  hyperbolic problems: applied and modern issues}, vol.~18, Elsevier, 2017.

\bibitem{abgrall2020neural}
{\sc R.~Abgrall and M.~H. Veiga}, {\em Neural network-based limiter with
  transfer learning}, Communications on Applied Mathematics and Computation,
  (2020), pp.~1--41.

\bibitem{archibald2005polynomial}
{\sc R.~Archibald, A.~Gelb, and J.~Yoon}, {\em Polynomial fitting for edge
  detection in irregularly sampled signals and images}, SIAM journal on
  numerical analysis, 43 (2005), pp.~259--279.

\bibitem{bacigaluppi2019posteriori}
{\sc P.~Bacigaluppi, R.~Abgrall, and S.~Tokareva}, {\em " a posteriori" limited
  high order and robust residual distribution schemes for transient simulations
  of fluid flows in gas dynamics}, arXiv preprint arXiv:1902.07773,  (2019).

\bibitem{beck2020neural}
{\sc A.~D. Beck, J.~Zeifang, A.~Schwarz, and D.~G. Flad}, {\em A neural network
  based shock detection and localization approach for discontinuous {G}alerkin
  methods}, Journal of Computational Physics, 423 (2020), p.~109824.

\bibitem{chan2018discretely}
{\sc J.~Chan}, {\em On discretely entropy conservative and entropy stable
  discontinuous {G}alerkin methods}, Journal of Computational Physics, 362
  (2018), pp.~346--374.

\bibitem{chen2020review}
{\sc T.~Chen and C.-W. Shu}, {\em Review of entropy stable discontinuous
  {G}alerkin methods for systems of conservation laws on unstructured simplex
  meshes}, CSIAM Transactions on Applied Mathematics, 1 (2020), pp.~1--52.

\bibitem{clain2011high}
{\sc S.~Clain, S.~Diot, and R.~Loub{\`e}re}, {\em A high-order finite volume
  method for systems of conservation laws—multi-dimensional optimal order
  detection (mood)}, Journal of computational Physics, 230 (2011),
  pp.~4028--4050.

\bibitem{Cybenko1989}
{\sc G.~Cybenko}, {\em Approximation by superpositions of a sigmoidal
  function}, Mathematics of Control, Signals and Systems, 2 (1989),
  pp.~303--314.

\bibitem{dafermos1973entropy}
{\sc C.~M. Dafermos}, {\em The entropy rate admissibility criterion for
  solutions of hyperbolic conservation laws}, Journal of Differential
  Equations, 14 (1973), pp.~202--212.

\bibitem{discacciati2020controlling}
{\sc N.~Discacciati, J.~S. Hesthaven, and D.~Ray}, {\em Controlling
  oscillations in high-order discontinuous {G}alerkin schemes using artificial
  viscosity tuned by neural networks}, Journal of Computational Physics, 409
  (2020), p.~109304.

\bibitem{du2016handbook}
{\sc Q.~Du, R.~Glowinski, M.~Hinterm{\"u}ller, and E.~Suli}, {\em Handbook of
  numerical methods for hyperbolic problems: basic and fundamental issues},
  Elsevier, 2016.

\bibitem{feireisl2019convergence}
{\sc E.~Feireisl, M.~Luk\'{a}\v{c}ov\'{a}-Medvid'ov\'{a}, and H.~Mizerov\'{a}},
  {\em Convergence of finite volume schemes for the {E}uler equations via
  dissipative measure-valued solutions}, Found. Comput. Math., 20 (2020),
  pp.~923--966, \url{https://doi.org/10.1007/s10208-019-09433-z},
  \url{https://doi.org/10.1007/s10208-019-09433-z}.

\bibitem{fisher2013discretely}
{\sc T.~C. Fisher, M.~H. Carpenter, J.~Nordstr{\"o}m, N.~K. Yamaleev, and
  C.~Swanson}, {\em Discretely conservative finite-difference formulations for
  nonlinear conservation laws in split form: {T}heory and boundary conditions},
  Journal of Computational Physics, 234 (2013), pp.~353--375.

\bibitem{glaubitz2019high}
{\sc J.~Glaubitz and A.~Gelb}, {\em High order edge sensors with {$l^1$}
  regularization for enhanced discontinuous {G}alerkin methods}, SIAM Journal
  on Scientific Computing, 41 (2019), pp.~A1304--A1330.

\bibitem{guermond2019invariant}
{\sc J.-L. Guermond, B.~Popov, and I.~Tomas}, {\em Invariant domain preserving
  discretization-independent schemes and convex limiting for hyperbolic
  systems}, Computer Methods in Applied Mechanics and Engineering, 347 (2019),
  pp.~143--175.

\bibitem{Harten83b}
{\sc A.~Harten}, {\em On the symmetric form of systems of conservation laws
  with entropy}, Journal of Computational Physics, 49 (1983), pp.~151--164.

\bibitem{Harten1989ENO}
{\sc A.~{Harten}}, {\em {ENO schemes with subcell resolution}}, {J. Comput.
  Phys.}, 83 (1989), pp.~148--184,
  \url{https://doi.org/10.1016/0021-9991(89)90226-X}.

\bibitem{ENOIII}
{\sc A.~Harten, B.~Enquist, S.~Osher, and S.~R. Chakravarthy}, {\em Uniformly
  high order accurate essentially non-oscillatory schemes {III}}, Journal of
  Computational Physics, 71 (1987), pp.~231--303.

\bibitem{HLL1983}
{\sc A.~Harten, P.~D. Lax, and B.~van Leer}, {\em On upstream differencing and
  {G}odunov type schemes for hyperbolic conservation laws}, 25 (1983),
  pp.~35--61.

\bibitem{hennemann2021provably}
{\sc S.~Hennemann, A.~M. Rueda-Ram{\'\i}rez, F.~J. Hindenlang, and G.~J.
  Gassner}, {\em A provably entropy stable subcell shock capturing approach for
  high order split form dg for the compressible euler equations}, Journal of
  Computational Physics, 426 (2021), p.~109935.

\bibitem{Isaacson1966Analyis}
{\sc E.~Isaacson and H.~B. Keller}, {\em Analysis of Numerical Methods}, Wiley,
  1966.

\bibitem{IsmailRoe2009}
{\sc F.~Ismail and P.~L. Roe}, {\em Affordable, entropy-consistent flux
  functions {II}: Entropy production at shocks}, Journal of Computational
  Physics, 228 (2009), pp.~5410--5436.

\bibitem{kingma2017adam}
{\sc D.~P. Kingma and J.~Ba}, {\em Adam: A method for stochastic optimization},
  2017, \url{https://arxiv.org/abs/1412.6980}.

\bibitem{klein2021using}
{\sc S.~Klein}, {\em Using the {D}afermos entropy rate criterion in numerical
  schemes}, arXiv preprint arXiv:arXiv:2202.13999,  (2022).

\bibitem{kuzmin2020monolithic}
{\sc D.~Kuzmin}, {\em Monolithic convex limiting for continuous finite element
  discretizations of hyperbolic conservation laws}, Computer Methods in Applied
  Mechanics and Engineering, 361 (2020), p.~112804.

\bibitem{kuzmin2021Limiter}
{\sc D.~{Kuzmin}, H.~{Hajduk}, and A.~{Rupp}}, {\em Limiter-based entropy
  stabilization of semi-discrete and fully discrete schemes for nonlinear
  hyperbolic problems}, {Comput. Methods Appl. Mech. Eng.}, 389 (2022), p.~28,
  \url{https://doi.org/10.1016/j.cma.2021.114428}.
\newblock Id/No 114428.

\bibitem{Lax71}
{\sc P.~D. Lax}, {\em Shock waves and entropy}, Contributions to Nonlinear
  Functional Analysis,  (1971), pp.~603--634.

\bibitem{lefloch2002fully}
{\sc P.~G. Lefloch, J.-M. Mercier, and C.~Rohde}, {\em Fully discrete, entropy
  conservative schemes of arbitrary order}, SIAM Journal on Numerical Analysis,
  40 (2002), pp.~1968--1992.

\bibitem{lukavcova2022convergence}
{\sc M.~Luk{\'a}{\v{c}}ov{\'a}-Medvi{\v{d}}ov{\'a} and P.~{\"O}ffner}, {\em
  Convergence of discontinuous {G}alerkin schemes for the {E}uler equations via
  dissipative weak solutions}, arXiv preprint arXiv:2202.10043,  (2022).

\bibitem{offner2015zweidimensionale}
{\sc P.~{\"O}ffner}, {\em Zweidimensionale klassische und diskrete orthogonale
  Polynome und ihre Anwendung auf spektrale Methoden zur L{\"o}sung von
  hyperbolischen Erhaltungsgleichungen}, PhD thesis, 2015.

\bibitem{offner2018stability}
{\sc P.~{\"O}ffner, J.~Glaubitz, and H.~Ranocha}, {\em Stability of correction
  procedure via reconstruction with summation-by-parts operators for {B}urgers'
  equation using a polynomial chaos approach}, ESAIM: Mathematical Modelling
  and Numerical Analysis, 52 (2018), pp.~2215--2245.

\bibitem{persson2006sub}
{\sc P.-O. Persson and J.~Peraire}, {\em Sub-cell shock capturing for
  discontinuous {G}alerkin methods}, in 44th AIAA Aerospace Sciences Meeting
  and Exhibit, 2006, p.~112.

\bibitem{ranocha2018comparison}
{\sc H.~Ranocha}, {\em Comparison of some entropy conservative numerical fluxes
  for the {E}uler equations}, Journal of Scientific Computing, 76 (2018),
  pp.~216--242.

\bibitem{ranocha2016summation}
{\sc H.~Ranocha, P.~{\"O}ffner, and T.~Sonar}, {\em Summation-by-parts
  operators for correction procedure via reconstruction}, Journal of
  Computational Physics, 311 (2016), pp.~299--328.

\bibitem{richtmyer1994difference}
{\sc R.~D. Richtmyer and K.~W. Morton}, {\em Difference methods for
  initial-value problems}, Malabar,  (1994).

\bibitem{Roe1981}
{\sc P.~L. Roe}, {\em Approximate riemann solvers, parameter vectors and
  difference schemes}, Journal of Computational Physics, 43 (1981),
  pp.~357--372.

\bibitem{rueda2022subcell}
{\sc A.~M. Rueda-Ram{\'\i}rez, W.~Pazner, and G.~J. Gassner}, {\em Subcell
  limiting strategies for discontinuous galerkin spectral element methods},
  arXiv preprint arXiv:2202.00576,  (2022).

\bibitem{Rumelhart1986Learning}
{\sc D.~E. Rumelhart, G.~E. Hinton, and W.~R. J.}, {\em Learning
  representations by back-propagating errors}, Nature,  (1986).

\bibitem{shi2018local}
{\sc C.~Shi and C.-W. Shu}, {\em On local conservation of numerical methods for
  conservation laws}, Computers \& Fluids, 169 (2018), pp.~3--9.

\bibitem{SO1988}
{\sc C.-W. Shu and S.~Osher}, {\em Efficient implementation of essentially
  non-oscillatory shock-capturing schemes}, Journal of Computational Physics,
  77 (1988), pp.~439--471.

\bibitem{SO1989}
{\sc C.-W. Shu and S.~Osher}, {\em Efficient implementation of essentially
  non-oscillatory shock-capturing schemesii}, Journal of Computational Physics,
  83 (1989), pp.~439--471.

\bibitem{sonntag2014shock}
{\sc M.~Sonntag and C.-D. Munz}, {\em Shock capturing for discontinuous
  {G}alerkin methods using finite volume subcells}, in Finite Volumes for
  Complex Applications VII-Elliptic, Parabolic and Hyperbolic Problems,
  Springer, 2014, pp.~945--953.

\bibitem{srivastava2014dropout}
{\sc N.~Srivastava, G.~Hinton, A.~Krizhevsky, I.~Sutskever, and
  R.~Salakhutdinov}, {\em Dropout: a simple way to prevent neural networks from
  overfitting}, The journal of machine learning research, 15 (2014),
  pp.~1929--1958.

\bibitem{Tadmor84II}
{\sc E.~Tadmor}, {\em Numerical viscsity and the entropy condition for
  conservative difference schemes}, Mathematics of Computation, 43 (1984),
  pp.~369--381.

\bibitem{tadmor1987numerical}
{\sc E.~Tadmor}, {\em The numerical viscosity of entropy stable schemes for
  systems of conservation laws. {I}}, Mathematics of Computation, 49 (1987),
  pp.~91--103.

\bibitem{zeifang2021data}
{\sc J.~Zeifang and A.~Beck}, {\em A data-driven high order sub-cell artificial
  viscosity for the discontinuous {G}alerkin spectral element method}, Journal
  of Computational Physics,  (2021), p.~110475.

\bibitem{zhang2011maximum}
{\sc X.~Zhang and C.-W. Shu}, {\em Maximum-principle-satisfying and
  positivity-preserving high-order schemes for conservation laws: survey and
  new developments}, Proceedings of the Royal Society A: Mathematical, Physical
  and Engineering Sciences, 467 (2011), pp.~2752--2776.

\end{thebibliography}

\end{document}